\documentclass[11pt]{article}
\usepackage{amsmath}
\usepackage{amssymb}
\usepackage{latexsym}
\usepackage{theorem}
\usepackage{mathrsfs}
\usepackage[all]{xy}

\newcommand{\smashprod}{\wedge}

\newcommand{\zz}{{ \mathbb{Z} }}
\newcommand{\nn}{{ \mathbb{N} }}
\newcommand{\qq}{{ \mathbb{Q} }}

\newcommand{\hh}{{ \mathbb{H} }}

\newcommand{\sphspec}{{ \mathbb{S} }}

\newcommand{\ccal}{{ \mathcal{C} }}
\newcommand{\dcal}{{ \mathcal{D} }}
\newcommand{\ecal}{{ \mathcal{E} }}
\newcommand{\mcal}{{ \mathcal{M} }}
\newcommand{\gcal}{{ \mathcal{G} }}
\newcommand{\fscr}{{ \mathscr{F} }}
\newcommand{\fcal}{{ \mathcal{F} }}

\newcommand{\fibrep}{{ \widehat{f} }}
\newcommand{\cofrep}{{ \widehat{c} }}

\DeclareMathOperator{\id}{Id}
\DeclareMathOperator{\h}{H}
\DeclareMathOperator{\ho}{Ho}

\DeclareMathOperator{\colim}{colim}

\DeclareMathOperator{\rightmod}{mod--}
\DeclareMathOperator{\leftmod}{--mod}

\DeclareMathOperator{\underhom}{\underline{Hom}}
\DeclareMathOperator{\sing}{Sing}

\DeclareMathOperator{\ch}{Ch}
\DeclareMathOperator{\gr}{gr.}

\DeclareMathOperator{\homSSS}{Hom}
\renewcommand{\hom}{\homSSS}

\newcommand{\sqq}{{ \text{s}\qq }}

\newtheorem{theorem}{Theorem}[section]
\newtheorem{proposition}[theorem]{Proposition}
\newtheorem{corollary}[theorem]{Corollary}
\newtheorem{lemma}[theorem]{Lemma}

\newtheorem{definition}[theorem]{Definition}
\newtheorem{remark}[theorem]{Remark}

\newenvironment{proof}[1][Proof]{\vskip -0.25cm \textbf{#1} }
{\hfill \rule{0.5em}{0.5em}}

\begin{document}

\title{Classifying Rational $G$-Spectra for Finite $G$}
\author{David Barnes}

\maketitle

\begin{abstract}
\noindent We give a new proof that for a finite group $G$, 
the category of rational $G$-equivariant spectra
is Quillen equivalent to the  
product of the model categories of chain complexes of 
modules over the rational group ring of 
the Weyl group of $H$ in $G$, 
as $H$ runs over the conjugacy classes of subgroups of $G$.
Furthermore the Quillen equivalences of our proof are all
symmetric monoidal. Thus we can understand categories
of algebras or modules over a ring spectrum in terms of 
the algebraic model. 
\end{abstract}

\section{Introduction}
A $G$-equivariant cohomology theory $E^*$ is said to be rational 
if $E^*(X)$ is a rational vector space for every $G$-space $X$.
For $G$, a finite group, we want to describe the category of rational 
$G$-equivariant cohomology theories
in terms of a simple algebraic model. 
In particular, we want to understand those cohomology theories
with a multiplication and the modules over such a theory. 
To do so, we give a particular construction of 
a model category of $G$-spectra 
whose homotopy category is (equivalent to)
the category of rational 
$G$-equivariant cohomology theories.
We show that this model category 
is symmetric monoidally Quillen equivalent to an explicit 
algebraic model: 
the product of the model categories of chain complexes of 
modules over the rational group ring of 
the Weyl group of $H$ in $G$, 
as $H$ runs over the conjugacy classes of subgroups of $G$.

Since our Quillen equivalences are symmetric monoidal, the category of ring spectra
is Quillen equivalent to the category of monoids 
in the algebraic model. 
Let $W_G H$ be the Weyl group of $H$ in $G$, 
the quotient of the normaliser of $H$ in $G$ by $H$, 
then a monoid in the category of chain complexes
of modules over the rational group ring of
$W_G H$, is a differential graded $\qq$-algebra with an 
action (through algebra maps) of 
the group $W_G H$. 
The category of modules over a ring spectrum
will then be Quillen equivalent to modules over a monoid in the algebraic model. 

So, if one has a ring spectrum $R$ that one wishes to study, one can look
at its image $\tilde{R}$ in the algebraic model and through the simplicity
of the model, perhaps obtain a good description of this monoid.
Here one can use the fact that the homology groups of $\tilde{R}$ are isomorphic
to the homotopy groups of $R$ to maintain some control of the homology
type of $\tilde{R}$.
In general, understanding the module category 
in the algebraic setting will then be a lot 
simpler than in the topological setting. 
Conversely one can take a differential graded algebra
with an action of the group $W_G H$ and this will correspond, 
via our Quillen equivalences,
to a ring spectrum in the category of $G$-spectra. 


We briefly describe the work
done previously in this area in 
order to introduce our result. 
The category of rational $G$-equivariant cohomology theories 
is equivalent to the 
product of the categories of graded
modules over the rational group ring of 
the Weyl group of $H$ in $G$, 
as $H$ runs over the conjugacy classes of subgroups of $G$.
see \cite[Appendix A]{gremay95} or \cite[Chapter XIX, Section 5]{may96}.

Let $G \mcal_\qq$ be a model category of $G$-spectra 
where the weak equivalences are those maps that induce isomorphisms
of rational homotopy groups.
Then the category of rational $G$-equivariant cohomology theories
is (equivalent to) the homotopy category of this model category.
The homotopy category of the 
model category of \emph{chain complexes} of 
modules over the rational group ring of $W_G H$ 
is isomorphic to the category of \emph{graded}
modules over the rational group ring of $W_G H$.
This follows since a rational chain complex is weakly
equivalent to its homology, with the latter treated
as a chain complex with all differentials zero.
Thus, it is natural to ask if the algebraic model
(the product of the model categories of chain complexes of 
modules over the rational group ring of $W_G H$, 
as $H$ runs over the conjugacy classes of subgroups of $G$) and 
$G \mcal_\qq$ are Quillen equivalent.
This was proven to be the case 
in \cite[Example 5.1.2]{ss03stabmodcat}. 

If one starts with a good model category of $G$-spectra,  
then one can ensure that $G \mcal_\qq$ has a 
smash product. Similarly, the tensor product
of rational chain complexes passes to a monoidal product on   
the algebraic model. 
In both cases, the monoidal product induces a product 
on the homotopy category.
So one can ask if these model categories are Quillen equivalent via a series of 
adjunctions which preserve the monoidal structures on the homotopy categories.
That is, is there a series of monoidal Quillen equivalences between 
$G \mcal_\qq$ and the algebraic model?
We construct such a series in this paper and thus we can conclude (using \cite{ss00}) 
that the various associated categories of algebras and modules
in these categories are Quillen equivalent. Our main results are presented in 
section \ref{sec:mainresults}.

\section{Outline of the proof}\label{sec:outline}

We begin by explaining \cite[Example 5.1.2]{ss03stabmodcat}:
starting from a category of rational $G$-spectra
one considers $\gcal_{top} = \{ \Sigma^\infty G/H_+ \}$, the set of 
suspension spectra of the orbit spaces $G/H_+$. 
This provides a set of generators
for the homotopy category of rational $G$-spectra. By using the good
properties of $L_\qq G \mcal$ we can construct
$\ecal_{top}(\sigma_1, \sigma_2)$, a symmetric spectrum of functions
for each pair, $(\sigma_1,\sigma_2)$, of objects of $\gcal_{top}$.
This collection has a composition rule,
$\ecal_{top}(\sigma_2, \sigma_3) \smashprod \ecal_{top}(\sigma_1, \sigma_2)
\to \ecal_{top}(\sigma_1, \sigma_3).$

Thus we have created an enriched category which we call
$\ecal_{top}$, it has object set $\gcal_{top}$ and the subscript
$top$ indicates that this category is of topological
origin. We can consider the category of enriched contravariant functors
from $\ecal_{top}$ to symmetric spectra. We call such
a functor a right $\ecal_{top}$-module
and these functors and the enriched natural transformations 
form a category $\rightmod \ecal_{top}$.
If $M$ is one of these enriched functors then
for each pair $(\sigma_1$, $\sigma_2)$ in $\gcal_{top}$
we have symmetric spectra
$M(\sigma_1)$ and $M(\sigma_2)$ with an action map
$M(\sigma_2) \smashprod \ecal_{top}(\sigma_1, \sigma_2) 
\to M(\sigma_1)$. An enriched natural transformation
$f \colon M \to N$ 
is then a collection of maps of symmetric spectra
$F(\sigma) \colon M(\sigma) \to N(\sigma)$
compatible with the action maps. If each
$f(\sigma)$ is a weak equivalence we say that $f$
is a weak equivalence of modules and these weak equivalences
are part of a model structure on $\rightmod \ecal_{top}$

This category of modules is referred to as the collection
of `topological Mackey functors' in \cite{ss03stabmodcat}.
The categories $L_\qq G \mcal$ and $\rightmod \ecal_{top}$
are Quillen equivalent by \cite[Theorem 3.3.3]{ss03stabmodcat}.
Since $G$ is finite and we are working
rationally, the homotopy groups of $\ecal_{top}(\sigma_1, \sigma_2)$
(that is, the set of graded maps in the homotopy category of symmetric spectra
from $S$ to $\ecal_{top}(\sigma_1, \sigma_2)$)
is concentrated in degree zero where it takes value
$\underline{A}(\sigma_1, \sigma_2)$,
a $\qq$-module. Hence $\ecal_{top}(\sigma_1, \sigma_2)$
it is weakly equivalent to an Eilenberg-MacLane spectrum
$\h \underline{A}(\sigma_1, \sigma_2)$.

From the collection of spectra $\h \underline{A}(\sigma_1, \sigma_2)$,
we construct a category $\h \underline{A}$, which is enriched
over symmetric spectra and one replaces $\rightmod \ecal_{top}$
by the Quillen equivalent category $\rightmod \h \underline{A}$.
The collection $\underline{A}(\sigma_1 , \sigma_2)$
for $\sigma_1, \sigma_2 \in \gcal_{top}$
can be thought of as defining category enriched over $dg \qq \leftmod$
and thus we have a model category $\rightmod \underline{A}$.
There is a zig-zag of Quillen equivalences
between $\rightmod \h \underline{A}$
and $\rightmod \underline{A}$.
Thus we have a zig-zag of Quillen equivalences between 
rational $G$-spectra and an algebraic category.

The algebraic category $\rightmod \underline{A}$ is actually well-known.  
Consider the collection of 
abelian group-enriched functors 
from $\underline{A}$ to $\qq \leftmod$, this is better known as the category
of rational Mackey functors, hence $\rightmod \underline{A}$
is the category of differential graded rational Mackey functors. 
The homotopy category of $\rightmod \underline{A}$ 
is equivalent to the category of graded rational Mackey functors 
which, by \cite[Appendix A]{gremay95}, 
is equivalent to $\prod_{(H) \leqslant G} \gr \qq W_G H \leftmod$.
This classification
does not consider monoidal structures since the zig-zag between 
$\rightmod \h \underline{A}$
and $\rightmod \underline{A}$ passes through a category without a
monoidal product. This proof relies upon the fact that 
the homotopy groups of $\ecal_{top}(\sigma_1, \sigma_2)$
are concentrated in degree zero.

We now turn to the method of this paper, 
in the next section we provide a diagram showing 
the Quillen equivalences that we use.
We start by replacing
$L_\qq G \mcal$ by the Quillen equivalent category 
$\prod_{(H) \leqslant G} S_H \leftmod$
(where each $S_H$ is a commutative ring spectrum) 
so that we can work one factor at a time.
The homotopy category 
of $S_H \leftmod$ is generated by 
$\sigma_H = G/H_+ \smashprod S_H$ and for any two $S_H$-modules 
we have a symmetric spectrum
function object $\ecal_{top}^H(X,Y)$. 
We define $\gcal_{top}^H$ to be all smash products of 
$\sigma_H$ including $S_H$ as the zero-fold smash.
Let $\ecal_{top}^H$ be the symmetric spectrum-enriched category 
with object set $\gcal_{top}^H$. The category $S_H \leftmod$ 
is Quillen equivalent to 
$\ecal_{top}^H \leftmod$.

One can now apply an alteration of \cite[Corollary 2.16]{shiHZ}
(included here as Proposition \ref{prop:shipmachine})
to construct a category $\ecal_t^H$ from
$\ecal_{top}^H$. This new category will be
enriched over rational chain complexes with its set
of objects given by $\gcal_{top}^H$. The $t$ indicates that we have come from
the topological side but are now working in an algebraic setting.
We can consider enriched functors from $\ecal_t^H$ to
rational chain complexes, this category will be
denoted $\rightmod \ecal_t^H$.

Now we begin our work from the other end, 
$\ch(\qq W_G H \leftmod)$ is generated by the object 
$\qq W_G H$, so we let $\gcal_a^H$ be the set of all
tensor products of $\qq W_G H$ (with $\qq$ as the 
zero-fold product). 
Analogously to the topological setting,
the set $\gcal_a^H$ is the object set for a category $\ecal_a^H$,
which is enriched over rational chain complexes.
We can then replace our algebraic model 
by the Quillen equivalent category $\rightmod \ecal_a^H$.
The $a$ indicates we have come from the algebraic model.

So far this process has been formal,
now we use some specific information about
$\ecal_t^H$ and $\ecal_a^H$ to achieve a comparison between them.
The comparison we will use is the notion of
a quasi-isomorphism of categories
enriched over rational chain complexes.
Given two such categories $\mathcal{C}$ and $\mathcal{D}$,
an enriched functor $F \colon \mathcal{C} \to \mathcal{D}$
is a quasi-isomorphism if 
$F$ induces an isomorphism on the classes of objects 
and each 
$F(\sigma_1, \sigma_2) \colon
\mathcal{C}(\sigma_1, \sigma_2) \to \mathcal{D}(F\sigma_1, F\sigma_2)$,
is a homology isomorphism
($F(\sigma_1, \sigma_2)$ is a map in the
category of rational chain complexes).
A quasi-isomorphism induces a Quillen equivalence
between $\rightmod \ccal$ and $\rightmod \dcal$.

In our case we define an isomorphism $F$ on the object sets
be sending the $i$-fold product of $\qq W_G H$ 
to the $i$-fold product of $G/H \smashprod S_H$. 
We then prove that 
the homology of $\ecal_t^H(F\sigma_1, F\sigma_2)$
is concentrated in degree zero and is isomorphic to 
$\ecal_a^H(\sigma_1, \sigma_2)$.
Since we are enriched over rational chain complexes, 
it follows that $\ecal_t^H$ and 
$\ecal_a^H$ are quasi-isomorphic
and hence that 
there is a zig-zag of Quillen equivalences between
$\rightmod \ecal_t^H$ and $\rightmod \ecal_a^H$.
Thus the category of rational $G$-spectra 
is Quillen equivalent to $\prod_{(H) \leqslant G} \ch(\qq W_G H \leftmod)$.
If we can prove that the quasi-isomorphisms between
$\ecal_t$ and $\ecal_a$ respect the monoidal structure, 
then it will follow that the zig-zag between 
rational $G$-spectra and the algebraic model
consists of monoidal Quillen equivalences. 

The difference between our method 
and that of \cite{ss03stabmodcat} is the ordering of the work.
We split the category, move to Mackey functors 
(modules over a symmetric spectrum-enriched category),
translate to algebra and then use the fact that 
$\ecal_t^H$ has homology concentrated
in degree zero.
Whereas, \cite{ss03stabmodcat} goes to Mackey functors first,
uses the fact that $\ecal_{top}$ has homotopy concentrated 
in degree zero, moves to algebra and then splits
the category.
The category $\ecal_t^H$ is constructed so that
its homology groups are isomorphic to the homotopy
groups of $\ecal_{top}^H$, hence the two methods
use the same information, just in different
contexts.

The result could have been proven
without using the splitting theorem.
The advantage to using this extra step is that 
$S_H \leftmod$ is generated by a single object,
which makes $\ecal_{top}^H$ and $\ecal_t^H$ 
easier to work with.

\section{Organisation}\label{sec:organise}

We describe the results of each section and 
display the Quillen equivalences that we use. 
Left adjoints will be placed on top, 
all adjunctions are symmetric (weak) monoidal
and furthermore all left adjoints except $L'$ 
are strong symmetric monoidal functors. 

In section \ref{sec:finitealg}, we examine our 
algebraic category and describe the Quillen equivalence between $\ch(\qq W_G H \leftmod)$ 
and $\rightmod \ecal_a^H$, a category of modules over an enriched category.
\[ 
(-) \otimes_{\ecal_a^H} \gcal_a^H \colon
\rightmod \ecal_a^H \overrightarrow{\longleftarrow}
\ch( \qq W_G H \leftmod ) \colon \underhom(\gcal_a^H, -).
\]


We then develop a category of rational spectra in section \ref{sec:finitetop}
and `split' this category into a product $\prod_{(H) \leqslant G} S_H \leftmod$
so that it looks more like the algebraic category. We end this section by 
replacing $S_H \leftmod$ by $\rightmod \ecal_{top}^H$, modules over a category enriched in
symmetric spectra, this step requires a little work. 
\[
G \mcal_\qq
\overset{\Delta}{\overrightarrow{\xleftarrow[\phantom{b}\prod\phantom{b}]{}}}
\prod_{(H) \leqslant G} L_{E \langle H \rangle} G \mcal_\qq
\overset{S_H \smashprod (-)}{\overrightarrow{\xleftarrow[\phantom{ab}U\phantom{ab}]{}}}
\prod_{(H) \leqslant G} S_H \leftmod
\overset{(-) \smashprod_{\ecal_{top}^H} \gcal_{top}^H}
{\overleftarrow{\xrightarrow[\underhom (\gcal_{top}^H,-)]{}}}
\prod_{(H) \leqslant G} \rightmod \ecal_{top}^H
\]
Now we give the functors from section \ref{sec:ETOPtoET}
(though we omit the identity adjunction of Lemma \ref{lem:postostable}).
The functors below are adaptations of those in \cite{shiHZ} and we define 
$\ecal_t^H$ to be $D \phi^* N \tilde{\qq} \ecal_{top}^H$, 
a category enriched over rational chain complexes.
\[
\rightmod \ecal_{top}^H
\overset{\tilde{\qq}}{\overrightarrow{\xleftarrow[\phantom{b}U'\phantom{b}]{}}}
\rightmod \tilde{\qq} \ecal_{top}^H
\overset{L'}
{\overleftarrow{\xrightarrow[\phi^* N]{}}}
\rightmod\phi^* N \tilde{\qq}  \ecal_{top}^H
\overset{D}{\overrightarrow{\xleftarrow[\phantom{b}R'\phantom{b}]{}}}
\rightmod D \phi^* N \tilde{\qq} \ecal_{top}^H
\]
In section \ref{sec:finitecomp} we prove that the homology category of the enriched
category $\ecal_t^H$ is isomorphic to the category $\ecal_a^H$. 
Let $\psi \colon \ecal_a^H \to \h_* \ecal_t^H$ be this isomorphism
(Theorem \ref{thm:hocalc})
and write $(\psi^{-1})^*$ for the left adjoint to 
$\psi^*$.
Let $i \colon C_0 \ecal_t^H \to \ecal_t^H $ and 
$p \colon C_0 \ecal_t^H \to \h_* \ecal_t^H$ be the maps constructed in 
Corollary \ref{cor:finiteETtoEA}. Then we have 
Quillen equivalences as below. 
\[
\rightmod \ecal_t^H
\overset{(-) \otimes_{C_0 \ecal_t^H} \ecal_t^H}
{\overleftarrow{\xrightarrow[\phantom{abcd}i^*\phantom{abcd}]{}}}
\rightmod C_0 \ecal_t^H
\overset{(-) \otimes_{C_0 \ecal_t^H} \h_* \ecal_t^H}
{\overrightarrow{\xleftarrow[\phantom{abcd}p^*\phantom{abcd}]{}}}
\rightmod \h_* \ecal_t^H
\overset{(\psi^{-1})^*}
{\overleftarrow{\xrightarrow[\phantom{ab}\psi^*\phantom{ab}]{}}}
\rightmod \ecal_a^H
\]
We put these results together to obtain the main result and 
consider categories of algebras and modules
in section \ref{sec:mainresults}.

We provide a list of the enriched categories that we use in this paper.
The object set $\gcal_a^H$ is the set of all tensor products of 
$\qq W_G H$ and $\gcal_{top}^H$ is the set of all smash products of 
$(\cofrep G/H_+ ) \smashprod S_H$ (in the category of $S_H$-modules).
For $S_H$-modules $X$ and $Y$, 
$\underhom (X,Y) = \sing \mathbb{U} (i^* \nn^\# F_{S_H}(X,Y))^{G}$
as we will explain later.
We define $\ecal_t^H$ to be $D \phi^* N \tilde{\qq} \ecal_{top}^H$.
Note that a category enriched over $\gr( \qq \leftmod)$
can also be considered as enriched over $\ch( \qq \leftmod)$.

\begin{tabular}{|c|c|c|c|}
\hline
Name & Enrichment & Object Set & Morphism Object \\
\hline 
\hline
\raisebox{-0.1cm}{$\ecal_a^H$} & 
\raisebox{-0.1cm}{$\ch( \qq \leftmod)$ }& 
\raisebox{-0.1cm}{$\gcal_a^H$ }& 
\raisebox{-0.1cm}{$\hom_\qq (a,b)^{W_G H}$}
\\[0.2cm]
\hline
\raisebox{-0.1cm}{$\ecal_{top}^H$ }& 
\raisebox{-0.1cm}{$Sp^\Sigma$ }& 
\raisebox{-0.1cm}{$\gcal_{top}^H$ }&  
\raisebox{-0.1cm}{$\underhom(X,Y)$ }
\\[0.2cm]
\hline
\raisebox{-0.1cm}{$\tilde{\qq} \ecal_{top}^H$ }& 
\raisebox{-0.1cm}{$Sp^\Sigma(\sqq \leftmod)$ }& 
\raisebox{-0.1cm}{$\gcal_{top}^H$ }& 
\raisebox{-0.1cm}{$\tilde{\qq} \underhom (X,Y)$ }
\\[0.2cm]
\hline
\raisebox{-0.1cm}{$\phi^* N \tilde{\qq} \ecal_{top}^H$ }& 
\raisebox{-0.1cm}{$Sp^\Sigma(\ch(\qq \leftmod)_+)$ }& 
\raisebox{-0.1cm}{$\gcal_{top}^H$ }&
\raisebox{-0.1cm}{$\phi^* N \tilde{\qq} \underhom (X,Y)$}
\\[0.2cm]
\hline
\raisebox{-0.1cm}{$\ecal_t^H$ }& 
\raisebox{-0.1cm}{$\ch( \qq \leftmod)$ }&
\raisebox{-0.1cm}{$\gcal_{top}^H$ }&
\raisebox{-0.1cm}{$D \phi^* N \tilde{\qq} \underhom (X,Y)$}
\\[0.2cm]
\hline
\raisebox{-0.1cm}{$\pi_* \ecal_{top}^H$ }& 
\raisebox{-0.1cm}{$\gr( \qq \leftmod)$ }&
\raisebox{-0.1cm}{$\gcal_{top}^H$ }& 
\raisebox{-0.1cm}{$\pi_* \underhom (X,Y)$}
\\[0.2cm]
\hline
\raisebox{-0.1cm}{$\ho \ecal_{top}^H$ }& 
\raisebox{-0.1cm}{$\gr( \qq \leftmod)$} &
\raisebox{-0.1cm}{$\gcal_{top}^H$ }& 
\raisebox{-0.1cm}{$[X,Y]^{S_H}_*$}
\\[0.2cm]
\hline
\raisebox{-0.1cm}{$\pi_* \gcal_{top}^H$ }& 
\raisebox{-0.1cm}{$\gr( \qq \leftmod)$ }&
\raisebox{-0.1cm}{$\gcal_{top}^H$ }& 
\raisebox{-0.1cm}{$\hom_\qq (\pi_*^H (X), \pi_*^H(Y))^{W_G H}$}
\\[0.2cm]
\hline
\raisebox{-0.1cm}{$C_0 \ecal_t^H$ }& 
\raisebox{-0.1cm}{$\ch( \qq \leftmod)$ }&
\raisebox{-0.1cm}{$\gcal_{top}^H$ }& 
\raisebox{-0.1cm}{$C_0 \ecal_t^H (X,Y)$}
\\[0.2cm]
\hline
\raisebox{-0.1cm}{$\h_* \ecal_t^H$ }& 
\raisebox{-0.1cm}{$\gr( \qq \leftmod)$ }&
\raisebox{-0.1cm}{$\gcal_{top}^H$ }& 
\raisebox{-0.1cm}{$\h_* \ecal_t^H (X,Y)$}
\\[0.2cm]
\hline
\end{tabular}

\textbf{Acknowledgments} This work consists of material from my PhD thesis, 
supervised by John Greenlees. I would to thank him for all the 
help and advice he has given me. I would also like to thank
Constanze Roitzheim for a careful reading 
of an earlier version of this paper
and for providing innumerable helpful suggestions. 

\section{The Algebraic Category}\label{sec:finitealg}

We study the algebraic model in detail
and the main result of this section is 
Theorem \ref{thm:FiniteAlgMorita}. 
In a very rough sense, this results puts the information 
of the algebraic model into a standard form, suitable for later comparisions.
We also use this section as an introduction
to the methods of this paper. 

For a ring $R$, let $\ch( R \leftmod)$ be the category of
chain complexes of left $R$-modules and let 
$\gr (R \leftmod)$ be the category of graded 
left $R$-modules.
The category of chain complexes of $R$-modules
has a model structure
(some times called the projective model structure) where
a map of chain complexes is a weak equivalence
if it is a homology isomorphism and
a fibration if it is a surjection.
The cofibrations are level-wise
split monomorphisms with cofibrant
cokernel. 
For each $n \in \zz$, let $S^{n}R$
be the chain complex concentrated
in degree $n$, where it takes value $R$.
Let $D^n R$ be the chain complex
with $R$ in degrees $n$ and $n-1$
and zeroes elsewhere, with
the identity as the differential
from degree $n$ to $n-1$.
The projective model structure is cofibrantly generated
with generating cofibrations
the inclusions $S^{n-1}R \to D^n R$
and generating acyclic cofibrations the maps 
$0 \to D^n R$. See
\cite[Section 2.3]{hov99} for more details.

For a finite group $G$, let $\qq G$ be the rational group
ring of $G$. This is a Hopf-algebra with co-commutative coproduct
$\Delta \colon \qq G \to \qq G \otimes \qq G$
induced by $g \mapsto g \otimes g$.
For $\qq G$-chain complexes $X$ and $Y$
we have $X \otimes_\qq Y$, the tensor product of 
$X$ and $Y$ considered as objects of $\ch(\qq \leftmod)$.
For $n \in \zz$, 
$(X \otimes_\qq Y)_n = \oplus_{i+j=n} X_i \otimes_\qq Y_j$
and we define a $G$-action by 
$g \cdot (x \otimes y) = (g \cdot x) \otimes (g \cdot y)$.
Hence $X \otimes_\qq Y$ is an object of $\ch(\qq G \leftmod)$.
That this product is associative and commutative
follows from the corresponding properties for 
the tensor product of $\ch(\qq \leftmod)$ and the 
co-commutative Hopf-algebra structure on $\qq G$. 
The unit of this product is $S^0 \qq$ equipped with
with trivial $G$-action.
Furthermore there is an internal homomorphism object, defined by
$\hom_\qq(X, Y)_n = \prod_k \hom_\qq(X_k, Y_{n+k})$,
 with $G$-action by conjugation 
and we have a natural isomorphism 
\[
\ch(\qq G \leftmod) (X \otimes_\qq Y, Z) \cong \ch(\qq G \leftmod) (X , \hom_\qq(Y,Z)).
\]

\begin{definition}
For any $X \in \ch( \qq G \leftmod)$,
there is natural map of chain complexes
$Av_G \colon X \to X^G$
defined by $Av_G(x) =  |G|^{-1} \Sigma_{g \in G} gx$.
\end{definition}

For a chain complex of $\qq$-modules $X$, let 
$\varepsilon^*(X)$ denote $X$ with the trivial action, 
an object of $\ch(\qq G \leftmod)$.
This functor is the left adjoint of a strong symmetric monoidal adjoint pair
(this terminology is defined later in this section)
\[
\varepsilon^* \colon \ch( \qq \leftmod) \overrightarrow{\longleftarrow}
\ch(\qq G \leftmod) \colon (-)^{G} 
\]
the right adjoint
is the fixed point functor.
We show that this is a Quillen pair by
proving that the right adjoint
preserves fibrations and weak equivalences.
Take $f \colon X \to Y$ a surjection
and let $y \in Y^G$, then there is an $x$ such that
$f(x) = y$. 
Let $Av_G(x) =  |G|^{-1} \Sigma_{g \in G} gx$, then since
$Av_G(x) \in X^G$
and $f(Av_G(x)) =Av_G(f(x))= Av_G(y) =y$, it follows
that $f^G$ is surjective.
That $(-)^G$ preserves weak equivalences
is immediate: $H_*(X^G) \cong (H_*X)^G$, as we are working rationally.
This implies that $\qq = \varepsilon^*(\qq)$ is 
cofibrant as an object of $\ch( \qq G \leftmod)$.

\begin{definition}\label{def:pushmonoid}
Let $\ccal$ be a cofibrantly generated model category, 
with a symmetric monoidal product $\otimes$, internal function object
$\hom_{\ccal}(-,-)$ and unit $I$.
Then $\ccal$ satisfies the \textbf{pushout product axiom} if 
the following three conditions hold (see \cite[Lemma 3.5(1)]{ss00}), in which
case $\ccal$ is called a \textbf{monoidal model category}.
\begin{enumerate}
\item If $f \colon A \to B$ and $g \colon C \to D$
are generating cofibrations then the pushout product,
$f \diamondsuit  g \colon B \otimes C \coprod_{A \otimes C} A \otimes D \to B \otimes D$, is a cofibration.
\item If $f$ is a generating cofibration and $g$
is a generating acyclic cofibration then
$f \diamondsuit g$ is a weak equivalence.
\item If $X$ is a cofibrant object then for any cofibrant 
replacement of the unit $\cofrep I \to I$ the induced map
$X \otimes \cofrep I \to X \otimes I$ is a weak equivalence. 
\end{enumerate}
Let $Z$ be any object of $\ccal$ and let 
$P_Z$ be the set of maps of the form 
$\id_Z \otimes f$ where $f$ is a generating acyclic
cofibration. The class $P_Z$-cell (\cite[Definition 2.1.9]{hov99}) 
is the collection of all maps
formed by transfinite compositions of pushouts of maps of 
$P_Z$. The model category $\ccal$ satisfies the \textbf{monoid axiom} if 
for any object $Z$ the class $P_Z$-cell consists of weak equivalences
(see \cite[Lemma 3.5(2)]{ss00}).
\end{definition}
The pushout product axiom ensures that the
monoidal product of a model category $\ccal$ induces a 
monoidal product on the homotopy category of 
$\ccal$. The monoid axiom (roughly speaking) ensures that 
there are model structures on the categories
of $R$-algebras and $R$-modules, for $R$ a 
commutative monoid in $\ccal$. 

\begin{proposition}\label{prop:chQGmonoid}
The tensor product and homomorphism object 
defined above gives the projective model structure on
$\ch( \qq G \leftmod)$
the structure of a closed symmetric monoidal model
category that satisfies the monoid axiom.
\end{proposition}
\begin{proof}
Let $f$ and $g$ be generating cofibrations,
then $f \diamondsuit g$ is an inclusion 
and the cokernel is $\qq(G \times G )$
(in some degree).
This cokernel is cofibrant: it is isomorphic
(as a $\qq G $-module) to $\bigoplus_{g \in G} \qq G$.
For a generating cofibration $f$ and a generating acyclic
cofibration $g$, 
$f \diamondsuit g$ is a weak equivalence since
both the domain and codomain are acyclic.
Since the unit $\qq$ is cofibrant the last condition
of the pushout product axiom holds automatically. 

The projective model structure on $ \ch( \qq \leftmod)$ 
satisfies the monoid axiom, this is proven for a general 
ring in \cite[Proposition 3.1]{shiHZ}. 
In fact this proof also suffices to show that 
$\ch( \qq G \leftmod)$ satisfies the monoid axiom, we copy that proof
with notation adjusted to our setting. 

The generating acyclic cofibrations for $\ch( \qq G \leftmod)$
are the maps $0 \to D^n(\qq G)$, for $n$ an integer.
Take any $Z \in \ch( \qq G \leftmod)$, then it is easy to check that
$Z \otimes_\qq  D^n(\qq G)$ is also acyclic.
Then we note that $0 \to Z \otimes_\qq D^n(\qq G)$
is an injection and a homology isomorphism.
Such maps are closed under pushouts and
transfinite compositions -- they are acyclic cofibrations
in the injective model structure (\cite[Theorem 2.3.13]{hov99}) on
chain complexes of $\qq G$-modules. 
Hence the monoid axiom holds for $\ch( \qq G \leftmod)$.
\end{proof}

The homotopy category of a pointed model category $\mathcal{C}$ supports
a suspension functor $\Sigma$ with a right adjoint loop functor $\Omega$,
see \cite[Section 6.1]{hov99}.
If these are inverse equivalences then $\mathcal{C}$ is 
called a \textbf{stable model category}.
All of the model categories that we use in this paper
are stable model categories. 
An object $X$ of $\mathcal{C}$ is said to be 
\textbf{compact} if for any family of objects $\{ Y_i \}_{i \in I}$, the canonical map
$\oplus_{i \in I} [X,  Y_i]^\mathcal{C} \to [X, \coprod_{i \in I} Y_i]^\mathcal{C}$,
is an isomorphism. A stable model category $\mathcal{C}$ 
is said to be \textbf{generated}
by a set of objects $\mathscr{P}$ if the smallest
full triangulated subcategory
of $\ho \mathcal{C}$ (with shift and triangles induced from $\ho \mathcal{C}$),
that is closed under coproducts, is $\ho \mathcal{C}$ itself. 
By \cite[Lemma 2.2.1]{ss03stabmodcat}, if the set $\mathscr{P}$
consists of compact objects, then this statement is
equivalent to the following: 
an object $X$ is trivial in the homotopy category
if and only if $[P, X]_*^\mathcal{C}$ (graded maps in the homotopy category) 
is zero for each $P \in \mathscr{P}$. 

\begin{lemma}
The model category $\ch( \qq G \leftmod)$
is generated by the compact object $\qq G$.
\end{lemma}
\begin{proof}
Let $X \in \ch( \qq G \leftmod)$, then 
$[\qq G, X]_*^{\qq G} \cong [\qq, X]^\qq \cong H_*(X)$.
\end{proof}

We now take the time to introduce the terminology 
of right modules over an enriched category and the notion
of monoidal Quillen equivalences. We will use 
this machinery in several different settings and 
it provides the framework for our method of proof. 
Later we will use other model categories in place of 
$\ch( \qq \leftmod)$. 

\begin{definition}
A $\ch( \qq \leftmod)$-\textbf{category} is a category enriched over 
$\ch( \qq \leftmod)$ (see \cite[Section 1.5]{kell05}).
A \textbf{right module} over a $\ch( \qq \leftmod)$-category $\ecal$
is a contravariant enriched functor $M \colon \ecal \to \ch( \qq \leftmod)$,
the category of such functors and enriched natural transformations 
is denoted by $\rightmod \ecal$. The \textbf{free module}
on an object $a$ of $\ecal$ is $F_a = \hom_\qq(-, a)$. 
\end{definition}
Let $M$ be a right $\ecal$-module, then for each object $a$ of $\ecal$
there is an object 
$M(a) \in \ch( \qq \leftmod)$.
For a pair of objects $a, b \in \ecal$ we have a map in 
$\ch( \qq \leftmod)$ 
\[
M_{a,b} \colon \ecal (a, b) \to \hom_\qq (M(b), M(a))
\]
a more useful version is given in terms of the adjoint to $M_{a,b}$, 
the `action map' 
$M(b) \otimes_\qq \ecal(a, b) \to M(a)$.
An enriched natural transformation 
$f \colon M \to N$ is a collection of maps
$f(a) \colon M(a) \to N(a)$ compatible with these action maps.
We can also form the category $\h_* \ecal$, this has the 
same object set as $\ecal$ and is enriched over graded 
$\qq$-modules, with morphism objects defined by 
$(\h_* \ecal)(a,b) = \h_* (\ecal(a,b))$. 

The category of right modules over $\ecal$ has a model 
structure with weak equivalences and fibrations 
defined object-wise in $\ch( \qq \leftmod)$,
see \cite[Subsection 3.3]{ss03stabmodcat}.
The collection of free modules is a generating set 
and these are cofibrant since the unit of 
$\ch( \qq \leftmod)$ is. 
The generating (acyclic) cofibrations of $\rightmod \ecal$
have form $A \otimes_\qq \id_{F_a} \to B \otimes_\qq \id_{F_a}$
(the object-wise tensor product)
for $A \to B$ a generating (acyclic) cofibration
of $\ch( \qq \leftmod)$.

Following \cite[Page 2]{day70} we define
a \textbf{symmetric monoidal enriched category} 
as an enriched category $\ecal$, 
with an enriched functor 
$\otimes \colon \ecal \times \ecal \to \ecal$ 
satisfying associativity, unitary and symmetry 
conditions. Such a category has 
a `unit object' which we denote by $I$. 
Thus for any two objects 
$a$ and $b$ of $\ecal$ we have an object $a \otimes b$ in $\ecal$ and 
for each quadruple $(a,b,c,d)$ of objects of $\ecal$
we have a map
\[
\ecal(a,c) \otimes_\qq \ecal(b,d) \to \ecal (a \otimes b, c \otimes d)
\]
which is compatible with the composition of $\ecal$.
The associativity, unitary and symmetry 
conditions imply that for any quadruple $(a,b,c,d)$
we have isomorphisms as below, which are compatible with 
the composition of $\ecal$.
\[
\begin{array}{rcl}
\ecal ((a \otimes b) \otimes c, d) \cong \ecal (a \otimes (b \otimes c), d) & & 
\ecal ( d, (a \otimes b) \otimes c) \cong \ecal ( d, a \otimes (b \otimes c)) \\
\ecal (a \otimes I,d) \cong \ecal (a,d) & &
\ecal (d, a \otimes I) \cong \ecal (d, a) \\
\ecal (a \otimes b , d) \cong \ecal (b \otimes a, d) &&
\ecal (d, a \otimes b ) \cong \ecal (d, b \otimes a)
\end{array}
\]
By assuming that $\ecal$ is a symmetric monoidal enriched category 
(and that the collection of objects of $\ecal$ forms a set)
we can put a symmetric monoidal structure on $\rightmod \ecal$
with unit the free module on $I$: $\ecal(-, I)$. 
The formula used is quite complicated, but it occurs 
often when constructing monoidal products.
The right-hand-side of the definition will be a coend, a 
particular form of a colimit and we give some details after the definition.
Let $M$ and $N$ be two 
objects of $\rightmod \ecal$ then their \textbf{box product}, $M \square N$,
is defined by the formula below. 
\[
M \square N(a)  = \int^{b,c} M(b) \otimes_\qq N(c) \otimes_\qq \ecal(a, b \otimes c)
\]
If $F \colon \ccal^{op} \times \ccal \to \ch( \qq \leftmod)$
is a $\ch( \qq \leftmod)$-enriched functor, where the objects of $\ccal$ form a set, then 
\[
\int^a F(a,a)  = \textrm{coeq} \left(
\coprod_{b, c \in \ccal}
F(b,c) \otimes_\qq \ccal (c,b) \overrightarrow{\longrightarrow}
\coprod_{d \in \ccal} F(d,d) \right).
\]
The two maps are given by $\ccal(c,b)$ acting on either the first or second
variable of $F$. We note that 
if $G$ is a right module over $\ccal$ then 
$\int^a  G(a) \otimes_\qq \ccal(b,a) \cong G(b)$
and if we have a functor 
\[
H \colon \ccal^{op} \times \ccal^{op} \times \ccal \times \ccal \to \ch( \qq \leftmod)
\]
then there is a canonical isomorphism 
$\int^a \int^b H(a,b,a,b) \cong \int^b \int^a H(a,b,a,b)$,
hence we allow ourselves to write $\int^{a,b} H(a,b,a,b)$ for either of these. 
The use of the functor 
$\otimes \colon \ecal \times \ecal \to \ecal$ is hidden in the definition
of $\square$, but becomes clear when $M \square N$ is written
in terms of a coequaliser of coproducts. For more
information on enriched categories and coends see \cite{mac} or \cite{bor94}.

The pushout product axiom and monoid axiom for $\ch(\qq \leftmod)$
imply that they also hold for $\rightmod \ecal$.
This routine statement is proven in \cite[Theroem 5.3.9]{barnes},
which also implies that the pushout product and monoid axioms hold
in all the categories of right modules that we will encounter.

As well as enrichments we can also consider tensorings and cotensorings over
the category $\ch( \qq \leftmod)$. That is, $\ccal$ 
is \textbf{tensored} over $\ch( \qq \leftmod)$
if it is equipped with a functor
\[
- \otimes - \colon \ch( \qq \leftmod) \times \ccal \to \ccal,
\]
unit isomorphisms $\qq \otimes a \cong a$
and associativity isomorphisms 
$(M \otimes_\qq N) \otimes a \cong M \otimes ( N \otimes a)$. 
Similarly $\ccal$ is \textbf{cotensored} over $\ch( \qq \leftmod)$
if there is a functor 
\[
\hom(- , -) \colon \ch( \qq \leftmod)^{op} \times \ccal \to \ccal
\]
also satisfying unital and associativity conditions. 
Often a category $\ccal$ will be 
enriched, tensored and cotensored over $\ch( \qq \leftmod)$ all at once,
whereupon for a chain complex $M$ and objects $a$ and $b$ of $\ccal$ 
we require isomorphisms of chain complexes as below
that make all the various unital and associativity conditions
compatible. 
\[
\hom_\qq (M, \ccal (a,b))
\cong
\ccal(M \otimes a, b)
\cong 
\ccal(a, \hom(M, b))
\]
One important consequence of having all three of these structures 
linked by isomorphisms as above is that this ensures that tensor operation preserves
colimits in both variables (similar statements then hold for the cotensor
and enrichment).

Let $\mathcal{C}$ be a model category which is enriched, tensored and cotensored 
over the category $\ch( \qq \leftmod)$, with isomorphisms 
relating the three structures as above.
Then $\mathcal{C}$ is said to be a $\ch( \qq \leftmod)$-\textbf{model category} 
(\cite[Definition 4.2.18]{hov99}) if whenever
$f$ is a cofibration of $\mathcal{C}$ and 
$g$ is a cofibration of $\ch( \qq \leftmod)$,
then $f \diamondsuit g$ 
(the notation $\diamondsuit$ is from Definition \ref{def:pushmonoid}) 
is a cofibration of $\mathcal{C}$
that is acyclic if one of $f$ or $g$ is. 

By definition a $\ch( \qq \leftmod)$-model category is a 
$\ch( \qq \leftmod)$-category. Note that 
$\ch( \qq G \leftmod)$ is a $\ch( \qq \leftmod)$-model category
with tensor, cotensor and enrichment defined via 
the adjunction $(\varepsilon^*,(-)^G)$. 
So for a $\qq$-chain complex $M$ and 
$\qq G$-chain complexes $X$ and $Y$, the tensor product is given by 
$M \otimes X = \varepsilon^*( M ) \otimes_\qq X$, 
the cotensor by 
$\hom_\qq (\varepsilon^*( M ), X)$ and the enrichment by 
$\hom_\qq (X,Y)^G$. 

For $H$ a subgroup of $G$,
let $N_G H$ be the normaliser of $H$: the largest subgroup of 
$G$ which contains $H$ as a normal subgroup, 
we then define $W_G H =  N_G H/H$, the Weyl group of $H$ in $G$. 
We write $\qq W_G H \leftmod$ for the category of 
$\qq$-modules with a left action of $W_G H$.

\begin{definition}
Let $\gcal_{a,G} = \{ \qq, \qq G, \qq (G \times G ), \qq(G \times G \times G), \dots \}$
and define $\ecal_{a,G}$ to be the $\ch( \qq \leftmod)$-category
with object set $\gcal_{a,G}$
and $\ch( \qq \leftmod)$-mapping object given by
$\ecal_{a,G} (X, Y) = \hom_\qq(X,Y)^{G }$.
Now we define $\gcal_{a,G}^H = \gcal_{a,W_G H}$
and $\ecal_{a,G}^H = \ecal_{a,W_G H}$. We will
usually suppress the $G$
and reduce this notation to
$\gcal_a^H$ and $\ecal_a^H$.
\end{definition}

Since the enrichment of $\ch( \qq W_G H \leftmod)$
over $\ch( \qq \leftmod)$ is defined in terms of the
strong symmetric monoidal adjunction $(\varepsilon^*,(-)^G)$,
the result below follows by a routine argument.

\begin{lemma}
The category $\ecal_a^H$ is a symmetric monoidal 
$\ch( \qq \leftmod)$-category.
\end{lemma}

By \cite[Theorem 3.9.3]{ss03stabmodcat}, the model 
categories $\ch( \qq W_G H \leftmod)$
and $\rightmod \ecal_a^H$ are Quillen equivalent, 
we describe the Quillen adjoint pair
of this result. The proof that these are an equivalence
is based on showing that the unit and counit of the derived
adjunction are isomorphisms on the generators
(the elements of $\gcal_a^H$ and the free modules).
Let $X$ be an object of $\ch( \qq W_G H \leftmod)$, then consider the functor
$\hom_\qq (-, X)^{W_G H} \colon \ecal_a^H \to \ch( \qq  \leftmod)$. 
This functor is enriched over $\ch( \qq  \leftmod)$ and thus we have
defined an object of $\rightmod \ecal_a^H$. 
Now let
\[
\underhom(\gcal_a^H, -) \colon \ch( \qq W_G H \leftmod) \to \rightmod \ecal_a^H
\]
be that functor which sends an object $X$ to the object $\hom_\qq (-, X)^{W_G H}$. 
This functor has a left adjoint, $(-) \otimes_{\ecal_a^H} \gcal_a^H$, 
defined in terms of a coend.
Let $M \in \rightmod \ecal_a^H$, then
\[
M \otimes_{\ecal_a^H} \gcal_a^H = \int^a M(a) \otimes a.
\] 
Now we show that the Quillen equivalence
between $\rightmod \ecal_a^H$ and
$\ch( \qq W_G H \leftmod)$ respects 
the monoidal structures. We first need some
terminology. 

\begin{definition}
Let $L \colon \ccal \overrightarrow{\longleftarrow} \ccal' \colon R$ 
be an adjunction between two monoidal categories
$(\ccal, \otimes, I)$ and $(\ccal', \otimes', I')$. 
Then $(L,R)$ is a \textbf{strong monoidal adjunction}
if the left adjoint is strong monoidal: 
so $LA \otimes' LB \cong L(A \otimes B)$, the units are related by 
an isomorphism $LI \cong I'$ and $L$ satisfies the 
associativity and unital coherence conditions of
\cite[Definition 4.1.2]{hov99}. 
Let $(L,R)$ be an adjunction of monoidal model categories such that 
there is a natural map 
$RX \otimes RY \to R(X \otimes' Y)$ and a specified map $I \to RI'$
which satisfy the associativity and unit conditions of
\cite[Diagrams 6.27 and 6.28]{bor94}.
This implies that the left adjoint has 
a natural map $m \colon L(A \otimes B) \to LA \otimes' LB$
and a map $LI \to I'$. 
We say that such an adjunction, $(L,R)$, 
is a \textbf{monoidal Quillen pair} 
(also known as lax monoidal or weak monoidal) 
if whenever $A$ and $B$ are cofibrant the map 
$m$ is a weak equivalence and 
if for any cofibrant replacement $\cofrep I \to I$
the composite $L( \cofrep I) \to LI \to I'$
is a weak equivalence. 
\end{definition}

If $(L,R)$ is a Quillen pair between monoidal model categories, 
such that $L$ is strong monoidal and the 
unit of the domain of $L$ is cofibrant, 
then $(L,R)$ is a monoidal Quillen pair. 
The conditions of a weak monoidal Quillen pair ensure that on homotopy categories
the derived adjunction is strong monoidal. In general a 
monoidal Quillen equivalence between monoidal model categories satisfying the monoid axiom
induces a Quillen equivalence on the categories of algebras and modules,
see \cite[Theorem 3.12]{ss03monequiv}. 

The adjoint pair $ ( (-) \otimes_{\ecal_a^H} \gcal_a^H, \underhom(\gcal_a^H, -))$,
is a strong symmetric monoidal Quillen adjunction, we prove part of this below,
see also \cite[Proposition 3.6]{greshi}. Let $M$ and $N$ be $\ecal_a^H$-modules, then

\begin{eqnarray*}
(M \square N) \otimes_{\ecal_a^H} \gcal_a^H
& = & \int^{g } \int^{a,b }
\big(
M(a) \otimes N(b) \otimes 
{\ecal(\mathcal{G})}(g, a \otimes b) \big)  \otimes g \\
& \cong & \int^{a,b } M(a) \otimes N(b) \otimes
\left( \int^{g }
{\ecal_a^H}(g, a \otimes b) \otimes g \right) \\
& \cong & \int^{a,b  } M(a) \otimes N(b) \otimes (a \otimes b) \\
& \cong & \int^{a,b } (M(a) \otimes a) \otimes (N(b) \otimes b) \\
& \cong & \int^{a } \left( (M(a) \otimes a)
\otimes \int^{b } (N(b) \otimes b) \right) \\
& \cong & \left( \int^{a } M(a) \otimes a \right)
\otimes \left( \int^{b } N(b) \otimes b \right)  \\
& = & M \otimes_{\ecal_a^H} \gcal_a^H
\otimes N \otimes_{\ecal_a^H} \gcal_a^H.
\end{eqnarray*}

We summarise the above work in the following result
which gathers all of the information of the category 
$\ch( \qq W_G H \leftmod )$ into a category of modules over 
a $\ch( \qq \leftmod )$-enriched category. 

\begin{theorem}\label{thm:FiniteAlgMorita}
There is a strong symmetric monoidal Quillen equivalence of symmetric
 mon\-oidal model categories
that satisfy the monoid axiom:
\[ 
(-) \otimes_{\ecal_a^H} \gcal_a^H \colon
\rightmod \ecal_a^H \overrightarrow{\longleftarrow}
\ch( \qq W_G H \leftmod ) \colon \underhom(\gcal_a^H, -).
\]
\end{theorem}

We now wish to repeat this operation for 
the model category of $G$-spectra and encode all of its information into
a $\ch( \qq \leftmod )$-category, this is a much more complicated task.
Once we have achieved this, we can compare 
this enriched category to $\ecal_a^H$.

\section{Rational $G$-Spectra and Splitting}\label{sec:finitetop}
We introduce a category of rational $G$-spectra
and use idempotents of the rational Burnside ring 
in Corollary \ref{cor:finitesplitting}
to split this category into a product 
of model categories, each generated
by a single object, indexed over the conjugacy 
classes of subgroups of $G$. 

We also provide a version
of this splitting in terms of modules
over a ring spectrum (Proposition \ref{prop:GspecHtoSHmod}).
We end this section with Theorem \ref{thm:finitemoritaequiv}, 
which performs the analogue of 
Theorem \ref{thm:FiniteAlgMorita}
for $S_H \leftmod$. 
We need to use the category $S_H \leftmod$ in this result
for technical reasons, 
as we explain in Remark \ref{rmk:whyfibrant}.
We take our time in introducing  
$S_H \leftmod$ as we need to understand
the weak equivalences and generators of these split pieces and 
it is easier to do so first, then move to modules
over ring spectra. 

The category of $G$-equivariant EKMM $S$-modules, $G \mcal$, is defined in 
\cite[Chapter IV]{mm02}, we refer to the objects of this category as $S$-modules
or $G$-spectra or just spectra. 
Let $H$ be a subgroup of $G$ and $n \geqslant 0$, then for an $S$-module $X$, 
we have the homotopy groups $\pi_n^H(X) = \pi_n (X(0)^H)$ 
and $\pi_{-n}^H(X) = \pi_n (X(\mathbb{R}^n)^H)$.
A map $f \colon X \to Y$ is called a $\pi_*$-isomorphism 
if $\pi_n^H(f)$ is an isomorphism for all integers $n$
and subgroups $H$ of $G$. See \cite[Chapter IV, Theorem 2.9]{mm02}
for the following result.  

\begin{theorem}\label{thm:ekmmmodel}
For $G$ a compact Lie group, 
there is a cofibrantly generated, proper, closed
symmetric monoidal model structure on
$G \mcal$\index{G M@$G\mcal$}
with weak equivalences the $\pi_*$-isomorphisms.
Every object of this category is fibrant. 
\end{theorem}

Let $E$ be a $G$-spectrum 
and let $X$ be a $G$-space, then we have a graded abelian group
$[ \Sigma^\infty X,E]^G_*$. This is the set of graded maps 
from the suspension spectrum of $X$ to $E$ in the homotopy 
category of $G$-spectra. 
We can think of this as a functor $E^*$ 
from the homotopy category of $G$-spaces 
to the category of graded abelian groups. 
The isomorphism classes of such functors as $E$ varies
is the category of $G$-equivariant cohomology theories. 
One could equally well give a direct definition of a $G$-cohomology theory
(\cite[Chapter XIII, Definition 1.1]{may96}) and
then prove that the category of such objects
is equivalent to the homotopy category of $G$-equivariant spectra.
If $E^*(X)$ is a $\qq$-module for every space $X$, 
then we say that $E^*$ is a rational $G$-cohomology theory. 
A map $f$ is called a \textbf{rational $\pi_*$-isomorphism}
(also called a \textbf{rational equivalence} or a $\pi_*^\qq$-isomorphism) if
$\pi_n^H(f) \otimes \qq$ is an isomorphism for all integers $n$
and subgroups $H$ of $G$. We now give a
result that summarises \cite[Section 2.2]{barnes}.
The homotopy category of the following model category is the category of
rational $G$-equivariant cohomology theories. 

\begin{theorem}\label{thm:ratGspec}
There is a cofibrantly generated, proper, closed
symmetric monoidal model structure on
the category of $G$-equivariant $S$-modules
with weak equivalences the 
$\pi_*^\qq$-isomorphisms, we denote this model structure by
$G \mcal_\qq$. Maps in the homotopy category
of $G \mcal_\qq$ will be written $[X,Y]^G_\qq$
and these sets are always rational vector spaces.
The fibrant objects are precisely those spectra with rational homotopy groups. 
\end{theorem}

Recall \cite[Chapter IV, Theorem 6.3]{mm02} which 
states that for a cofibrant spectrum $E \in G \mcal$, 
there is an \textbf{$E$-model structure} on the category of 
$G$-equivariant $S$-modules with the same cofibrations as before 
and weak equivalences those maps $f$ such that 
$f \smashprod \id_E$ is a $\pi_*$-isomorphism. 
This model structure is called the Bousfield 
localisation of $G \mcal$ at $E$ and is written $L_E G \mcal$,
the fibrant objects of this category are precisely
the $E$-local objects. Note that in terms of model categories
$L_{E \smashprod F} G \mcal = L_E L_F G \mcal = L_F L_E G \mcal$,
that is, the weak equivalences, cofibrations and fibrations are the same.

We construct $G \mcal_\qq$ by localising at $S^0_\mcal \qq$, 
a `rational sphere spectrum'.
This is a cofibrant spectrum such that the 
$S^0_\mcal \qq$-equivalences are the $\pi_*^\qq$-isomorphisms.
This spectrum is constructed as follows, using $\cofrep$ to
denote cofibrant replacement. 
Take $0 \to \oplus_i \zz \overset{f}{\to} \oplus_j \zz \to \qq \to 0$, 
a free resolution of $\qq$ as a $\zz$-module. 
Choose a map $g \colon \vee_i \cofrep S \to \vee_j \cofrep S$
such that $\pi_0^G(g)$ is given by 
$f \otimes \id \colon \oplus_i \zz \otimes A(G) \to \oplus_j \zz \otimes A(G)$.
The spectrum $S^0_\mcal \qq$ is then defined to 
be the cofibre of $g$. With these definitions we can now give
\cite[Theorem 3.2.4]{barnes}, which we will use to split the category of 
rational $G$-spectra into more manageable pieces. 

\begin{theorem}\label{thm:split}
Let $\{ E_i \}_{i\in I}$ be a finite collection of cofibrant orthogonal $G$-spectra or
$G$-spaces. If $E_i \smashprod E_j$ is rationally acyclic for $i \neq j$ and 
$\bigvee_{i \in I} E_i$ is rationally equivalent
to $S$, then we have a strong symmetric monoidal Quillen equivalence
\[
\Delta \colon G \mcal_\qq
\overrightarrow{\longleftarrow}
\prod_{i \in I} L_{E_i } G \mcal_\qq
\colon \prod.
\]
The left adjoint takes a $G$-spectrum $X$
to the constant collection of $X$ in each factor.
The right adjoint takes the collection
$\{ Y_i \}_{i \in I}$ to the $G$-spectrum $\prod_{i \in I} Y_i $.
\end{theorem}

An important step in the proof of this theorem is 
showing that if $X$ is $E_i$-local then $E_j \smashprod X \to *$
is a rational equivalence whenever $i \neq j$. 
We will need this later and in fact this result can be deduced 
from the above theorem. 

Take $X$ an $E_i$-local $G$-spectrum, then the collection
$\{ X_i \}_{i \in I}$ defined by $X_i = X$ and $X_j = *$
whenever $i \neq j$, is a fibrant object of 
$\prod_{i \in I} L_{E_i } G \mcal_\qq$.
Since $\Delta$ preserves all weak equivalences and $(\Delta, \prod)$ 
is a Quillen equivalence, it follows 
that the counit is a weak equivalence. Hence $X$ is $E_j$-equivalent
to $*$ whenever $i \neq j$. 

The Burnside ring of $G$, $A(G)$, is the Grothendieck ring
of finite $G$-sets and is isomorphic to 
$[S,S]^G_*$. Since $G$ is finite, tom-Dieck's isomorphism 
(see \cite[Chapter V, Lemma 2.10]{lms86})
specifies an isomorphism 
$A(G) \otimes \qq \cong \prod_{(H)\leqslant G} \qq$.
Thus, for each conjugacy class of subgroups, $(H) \leqslant G$,
there is an idempotent $e_H \in A(G) \otimes \qq$ 
given by projection onto factor $(H)$.
Let $\fibrep_\qq$ denote fibrant replacement
in $G \mcal_\qq$, then 
$A(G) \otimes \qq \cong [\fibrep_\qq S, \fibrep_\qq S]^G$.
Given an idempotent $e$ in the rational Burnside ring 
we write $e S$ for the homotopy colimit (telescope) of 
$S \to \fibrep_\qq S \overset{f}{ \to} \fibrep_\qq S \overset{f}{ \to} \dots$,
for some representative $f$ of $e$. Using the diagram
$
X \to X \smashprod \fibrep_\qq S \overset{\id \smashprod f}{\longrightarrow} 
X \smashprod \fibrep_\qq S \overset{\id \smashprod f}{\longrightarrow} \dots
$
we construct $e X$ for any spectrum $X$.
The map $X \to X \smashprod \fibrep_\qq S$
is a $\pi_*^\qq$-isomorphism. Hence $e$ (or rather $\id_X \smashprod f$)
induces a self-map of $\pi_*^H(X) \otimes \qq$, we write this map as $\iota^*_H(e)_*$.
Homotopy groups and idempotents commute in the sense that the canonical map
$\iota^*_H(e)_* \pi_*^H(X) \otimes \qq \to \pi_*^H(e X) \otimes \qq$ 
is an isomorphism. 

\begin{definition}
For a group $G$, with subgroups $H$ and $K$,
we say that $K$ is \textbf{subconjugate}\index{Subconjugate}
to $H$ if the $G$-conjugacy class of $K$ contains
a subgroup of $H$, we write $K \leqslant_G H$. In turn $K$ is
\textbf{strictly subconjugate}\index{Strictly subconjugate}
to $H$ if the $G$-conjugacy class of $K$ contains
a strict subgroup of $H$, the notation for this is $K <_G H$.
\end{definition}

\begin{definition}A set of subgroups of $G$ is called a \textbf{family} if 
it is closed under conjugation and taking subgroups.
For each family $\fcal$ there is a $G$-CW complex 
$E \fcal$ which satisfies the universal property:
$E \fcal^H$ is contractible for $H \in \fcal$
and is empty otherwise. The cofibre of the 
projection map $E \fcal_+ \to S^0$ is denoted by 
$\widetilde{E} \fcal$.
\end{definition}

Take $H$ a subgroup of $G$, then we have a pair of families
of subgroups of $G$:
$[\leqslant_G H]$ -- the family of all subgroups of $G$
which are subconjugate to $H$
and $[<_G H]$ -- the family of all subgroups of $G$
which are strictly subconjugate to $H$.
We can then form $G$-CW complexes
$E[\leqslant_G H]_+$\index{E [ H]@$E[\leqslant_G H]_+$}
and $E[<_G H]_+$\index{E [ H]@$E[<_G H]_+$}.
There is a map $E[<_G H]_+ \to E[\leqslant_G H]_+$, we call the cofibre
of this map $E\langle H\rangle $\index{E \lange H \rangle @$E \langle H \rangle$}.

Note that since $E[<_G H]_+$ and $E[\leqslant_G H]_+$
are cofibrant as $G$-spaces, the space $E\langle H\rangle$
is also cofibrant as a $G$-space. We can also describe $E\langle H\rangle$
as $E[\leqslant_G H]_+ \smashprod \widetilde{E} [<_G H]$.
Since geometric fixed point functors preserve cofibre sequences, the spectrum
$\Phi^K (\Sigma^\infty E\langle H\rangle)$ is contractible unless
$(K)=(H)$, whence it is non-equivariantly
rationally equivalent to $S$.
The following is a standard result proven by looking at geometric fixed points
(see \cite[Lemma 3.4.11]{barnes}).

\begin{lemma}
Let $e_{[\leqslant_G H]} = \Sigma_{(K) \leqslant H} e_K$ and
$e_{[<_G H]} = \Sigma_{(K) < H} e_K$.
Then there are zig-zags of rational 
$\pi_*$-isomorphisms between $E [\leqslant_G H]_+$ and
$e_{[\leqslant_G H]} S$ and similarly so for $E [<_G H]_+$ and
$e_{[<_G H]} S$. Furthermore $E \langle H \rangle$ is rationally
equivalent to $e_H S$.
\end{lemma}

From this it follows that a map $f \colon X \to Y$ in $G \mcal_\qq$ 
is a rational $E \langle H \rangle$-equivalence
if and only if $e_H f \colon e_H X \to e_H Y$ is a rational equivalence.
We can now apply the splitting theorem using the
set of objects $E \langle H \rangle$ as $H$ runs over a set of representatives
for the conjugacy classes of subgroups of $G$. 
Since $E \langle H \rangle$ is rationally
equivalent to $e_H S$ it follows that $\bigvee_{(H) \leqslant G} E \langle H \rangle$
is rationally equivalent to $S$ and $E \langle H \rangle \smashprod E \langle K \rangle$
is rationally acyclic whenever $H$ and $K$ are not conjugate.

\begin{corollary}\label{cor:finitesplitting}
There is a strong symmetric monoidal Quillen equivalence between the category
of rational $G$-spectra and the product of the categories
$L_{E \langle H \rangle} G \mcal_\qq$, as $H$ runs over the set of conjugacy
classes of subgroups of $G$.
\[
\Delta \colon G \mcal_\qq
\overrightarrow{\longleftarrow}
\prod_{(H) \leqslant G} L_{E \langle H \rangle} G \mcal_\qq
\colon \prod
\]
\end{corollary}

\begin{lemma}\label{lem:modelequality}
There is an equality of model structures:
\[
L_{E \langle H \rangle} G \mcal_\qq =
L_{E \langle H \rangle} L_{E [ \leqslant_G H]_+ } G \mcal_\qq 
\]
that is to say, the weak equivalences, cofibrations and fibrations
agree.
\end{lemma}
\begin{proof}
The cofibrations of these two model structures agree by definition.
The map $\widetilde{E}  [ \leqslant_G H] \to *$ is a rational
$E \langle H \rangle$ equivalence.
Hence, considering the cofibre sequence
which defines $\widetilde{E}  [ \leqslant_G H]$
we have a rational equivalence
$E [ \leqslant_G H]_+ \smashprod E \langle H \rangle \to E \langle H \rangle.$
It follows that a rational
$E [ \leqslant_G H]_+ \smashprod E \langle H \rangle$-equivalence
is a rational $E \langle H \rangle$-equivalence. So the weak equivalences
of $L_{E \langle H \rangle} G \mcal_\qq$ and
$L_{E \langle H \rangle} L_{E [ \leqslant_G H]_+ } G \mcal_\qq $
agree. 
\end{proof}

By \cite[IV, Proposition 6.7]{mm02}, the weak equivalences of
$L_{E [ \leqslant_G H]_+ } G \mcal_\qq $
are those maps $f$ such that $\pi_*^K(f) \otimes \qq$
is an isomorphism for all $K \leqslant_G H$.

\begin{lemma}\label{lem:WElocal}
A map $f$ in $L_{E \langle H \rangle} G \mcal_\qq $
is a weak equivalence if and only if the induced map of homotopy groups
$\iota_H^*(e_H)_* \pi_*^H(f) \otimes \qq$ is a isomorphism.
Hence $G/H_+$
is a compact generator for $L_{E \langle H \rangle} G \mcal_\qq $.
\end{lemma}
\begin{proof}
Lemma \ref{lem:modelequality} shows that $f$
is a weak equivalence if and only if
$\iota_K^*(e_H)_* \pi_*^K(f) \otimes \qq$ is an isomorphism for
all $K \leqslant_G H$.
For any $G$-spectrum $X$, the set $\pi_*^K(X) \otimes \qq$ 
is a module over $A(K) \otimes \qq$. 
The rational Burnside ring of $G$ acts 
on $\pi_*^K(X) \otimes \qq$ via the restriction map
$\iota_K^* \colon A(G) \otimes \qq \to A(K) \otimes \qq$
induced from the inclusion $\iota \colon K \to G$.
Now note that if $K$ is a strict subgroup of $H$ then
$\iota_K^*(e_H)=0$, hence for any map $f$,
$\iota_K^*(e_H)_* \pi_*^K(f) \otimes \qq$ will be an isomorphism. 
This proves the first statement. 

For any $G$-spectrum $X$, $e_H \pi_*^H (X) \otimes \qq$  
is isomorphic to $\pi_*^H(\fibrep_H X)$, 
where $\fibrep_H X$ is the fibrant replacement of
$X$ in $L_{E \langle H \rangle} G \mcal_\qq$, 
hence $\fibrep_H X$ is $S^0_\mcal \qq \smashprod E \langle H \rangle$-local.
In turn, $\pi_*^H(\fibrep_H X)  \cong [G/H_+,\fibrep_H X]^G_*$
which is isomorphic to 
$([G/H_+, X]^{G|H}_\qq)_*$, the set of graded maps in the homotopy category of 
$L_{E \langle H \rangle} G \mcal_\qq$.
We have already shown that $X$ is trivial in 
$\ho L_{E \langle H \rangle} G \mcal_\qq $ if and only if
$e_H \pi_*^H (X) \otimes \qq =0$. Now we know that 
$X$ is trivial in 
$\ho L_{E \langle H \rangle} G \mcal_\qq $ if and only if
$([G/H_+, X]^{G|H}_\qq)_* =0$, hence $G/H_+$ is
 a generator, it is compact since $G/H_+$
is a compact space. 
\end{proof}

Our next task is to obtain a version of 
$L_{E \langle H \rangle} G \mcal_\qq $
with every object 
fibrant (see Remark \ref{rmk:whyfibrant}).

\begin{lemma}\label{lem:SHobject}
There is an $S^0_\mcal \qq \smashprod E \langle H \rangle$-local
commutative $S$-algebra
$S_H$\index{SH@$S_H$} whose
unit map is a rational $E \langle H \rangle$-equivalence.
Furthermore every $S_H$-module is
$S^0_\mcal \qq \smashprod E \langle H \rangle$-local.
\end{lemma}
\begin{proof}
This result is an application of 
\cite[Chapter VIII, Theorem 2.2]{EKMM97}
which is easily adapted to an equivariant setting. 
We use the cell object $S^0_\mcal \qq \smashprod E \langle H \rangle$
to create a commutative cell $S$-algebra $S_H$ which is the
$S^0_\mcal \qq \smashprod E \langle H \rangle$-localisation of $S$.
By construction, the unit map $S \to S_H$ is a
rational $E \langle H \rangle$-equivalence, hence
$S^0_\mcal \qq \smashprod E \langle H \rangle$ is $\pi_*$-isomorphic
to $S_H \smashprod S^0_\mcal \qq \smashprod E \langle H \rangle$.
Since $S_H$ is $S^0_\mcal \qq$-local, it has rational homotopy groups,
thus there is a
zig-zag of weak equivalences
$S^0_\mcal \qq \smashprod S_H \leftarrow \cofrep S \smashprod S_H \to S_H$.
Equally $S_H$ is weakly equivalent to
$S_H \smashprod \bigvee_{(K)} E \langle K \rangle$.
Since $S_H$ is $E \langle H \rangle$-local,
$S_H \smashprod E \langle K \rangle$ is acyclic
whenever $(H) \neq (K)$ (as noted above, this is part of the proof of
Theorem \ref{thm:split}).
It follows that $S_H \smashprod \bigvee_{(K)} E \langle K \rangle$
is weakly equivalent to $S_H \smashprod E \langle H \rangle$.
Thus $S_H$ is $\pi_*$-isomorphic to $S^0_\mcal \qq \smashprod E \langle H \rangle$.
The rest of the result is standard, see
\cite[13.1]{adams}.
\end{proof}

\begin{proposition}\label{prop:GspecHtoSHmod}
The adjoint pair of the free $S_H$-module functor
and the forgetful functor
\[
S_H \smashprod (-) \colon L_{E \langle H \rangle} G \mcal_\qq
\overrightarrow{\longleftarrow}
S_H \leftmod \colon U 
\]
is a strong symmetric monoidal Quillen equivalence.
\end{proposition}
\begin{proof}
This is easy to prove, the two points to note are: an
$S^0_\mcal \qq \smashprod E \langle H \rangle$-equivalence
between $S^0_\mcal \qq \smashprod E \langle H \rangle$-local objects
is a $\pi_*$-isomorphism and the unit map of $S_H$ is
an $S^0_\mcal \qq \smashprod E \langle H \rangle$-equivalence.
\end{proof}

We now fix a cofibrant replacement of the suspension
spectrum of $G/H_+$. We call this 
$\cofrep G/H_+$. One example is given by 
$S \smashprod_{\mathscr{L}} \mathbb{L} \Sigma^\infty G/H_+$,
see \cite[Chapter IV, Proposition 2.1]{mm02}.

\begin{lemma}
The object $(\cofrep G/H_+) \smashprod S_H$
is a $G$-compact, cofibrant and fibrant generator
of $S_H \leftmod$.
\end{lemma}
\begin{proof}
Every object of $S_H \leftmod$ is fibrant and since
$\cofrep G/H_+$ is a cofibrant spectrum, 
$(\cofrep G/H_+ ) \smashprod S_H$ is cofibrant in $S_H \leftmod$.
This object is $G$-compact since the right adjoint $U$ commutes with
filtered colimits and $G/H_+$ is a $G$-compact $G$-spectrum.
Since $G/H_+$ generates $L_{E \langle H \rangle} G \mcal_\qq $, 
which is Quillen equivalent to $S_H \leftmod$, 
it follows that $S_H \leftmod$ is generated by 
$(\cofrep G/H_+ ) \smashprod S_H$.
\end{proof}

Now we perform the analogue of Theorem \ref{thm:FiniteAlgMorita}
for $S_H \leftmod$. 
Recall the positive model structure
as defined on symmetric spectra (and other categories of 
diagram spectra), written $Sp^\Sigma_+$, from 
\cite[Theorem 14.1]{mmss01}.
The positive model structure has the same weak equivalences as $Sp^\Sigma$
but the unit is no longer cofibrant. The identity functor is the left
adjoint of a Quillen equivalence from $Sp^\Sigma_+$ to $Sp^\Sigma$.
The adjunction $(\nn, \nn^{\#})$ below, is only a Quillen pair when 
we use the positive model structure on equivariant
orthogonal spectra ($G \mathscr{IS}^{U}_+$), 
hence all the other categories below must be given their
positive model structures.  

We must first prove that 
$S_H \leftmod$ is an $Sp^\Sigma_+$-model category, we do so
by constructing a strong symmetric monoidal
Quillen adjunction with left adjoint mapping from $Sp^\Sigma_+$
to $S_H \leftmod$. 
To do so we must be careful about change of universe functors 
since we need these to be both strong monoidal and compatible with the model structures. 
To solve this we pass through equivariant orthogonal spectra, 
also defined in \cite{mm02}. Thus we have the following unwieldy
series of adjoint pairs.

The adjunction of geometric realisation and the singular complex functor
(between simplicial sets and topological spaces)
induces the Quillen equivalence below,
where $Sp^\Sigma (\textrm{Top})_+$ is the category of symmetric spectra of topological
spaces with the positive model structure.
\[
|-| \colon Sp^\Sigma_+ \overrightarrow{\longleftarrow}
Sp^\Sigma (\textrm{Top})_+ \colon \sing
\]
One can then prolong to the positive model structure on orthogonal spectra 
(indexed on the universe $\mathbb{R}^\infty$), which we write as
$\mathscr{IS}^{\mathbb{R}^\infty}_+$, using 
the Quillen equivalence of \cite[Theorem 10.4]{mmss01}.
\[
\mathbb{P} \colon Sp^\Sigma (\textrm{Top})_+ \overrightarrow{\longleftarrow}
\mathscr{IS}^{\mathbb{R}^\infty}_+ \colon \mathbb{U}
\]
The trivial action and fixed point adjunction (\cite[Chapter V, Section 3]{mm02}) 
move us to $G$-equivariant orthogonal spectra 
indexed on a trivial universe.
\[
\varepsilon^* \colon \mathscr{IS}^{\mathbb{R}^\infty}_+ \overrightarrow{\longleftarrow}
G \mathscr{IS}^{\mathbb{R}^\infty}_+ \colon (-)^G
\]
We apply change of universe functors to move 
to a complete universe $U$, using the notation of 
\cite[Chapter V, Proposition 3.4]{mm02}.
\[
i_* \colon G \mathscr{IS}^{\mathbb{R}^\infty}_+ \overrightarrow{\longleftarrow}
G \mathscr{IS}^{U}_+ \colon i^*
\]
We can relate the above to EKMM spectra using the 
Quillen equivalence of \cite[Chapter IV, Theorem 1.1]{mm02}.
\[
\nn \colon G \mathscr{IS}^{U}_+ \overrightarrow{\longleftarrow}
G \mcal^{U} \colon \nn^{\#}
\]
Then we make use of the free $S_H$-module functor.
\[
- \smashprod S_H \colon G \mcal^U  \overrightarrow{\longleftarrow}
S_H \leftmod \colon U
\]
Since each of these adjoint pairs is strong symmetric monoidal 
it follows that $S_H \leftmod$ is an $Sp^\Sigma_+$-model category, 
the enrichment is given by the following formula. Let $X$ and $Y$ be 
$S_H$-modules, then $\underhom (X,Y) = \sing \mathbb{U} (i^* \nn^\# F_{S_H}(X,Y))^{G}$
is the symmetric spectrum object of functions from $X$ to $Y$. 
That is, one takes the function spectrum of $S_H$-modules, 
applies the functor $\nn^\#$, moves to a trivial universe ($i^*$),
takes $G$-fixed points, moves down to symmetric spectra of 
topological spaces and applies the singular complex functor
to get to $Sp^\Sigma_+$. This construction comes equipped with a
natural isomorphism 
$\pi_*(\underhom(X,Y)) \cong [X,Y]^{S_H}_*$. 

\begin{definition}\label{def:finitegcalH}
Let $\gcal_{top}^H$ be the set of all smash products 
(in the category of $S_H$-modules) of
$(\cofrep G/H_+ ) \smashprod S_H$, 
we include $S_H$ as the zero-fold smash product.
Let $\ecal_{top}^H$ be 
the $Sp^\Sigma_+$-enriched category on
the objects of $\gcal_{top}^H$.
\end{definition}

The enrichment of $S_H \leftmod$ over $Sp^\Sigma_+$
is defined in terms of a (series of) strong symmetric
monoidal adjunctions. Hence it is routine to prove that
$\ecal_{top}^H$ is a symmetric monoidal enriched category. 
With the exception of the unit, all objects of $\gcal_{top}^H$
are cofibrant and all objects are fibrant.
We replace the category of $S_H$-modules
by the Quillen equivalent category of modules 
over $\ecal_{top}^H$, see also \cite[Proposition 4.1]{greshi}.
Thus we have encoded $S_H \leftmod$ in terms of a 
symmetric spectrum-enriched category. 
The adjunction below is analogous to the functors of 
Theorem \ref{thm:FiniteAlgMorita}.

\begin{theorem}\label{thm:finitemoritaequiv}
The adjoint pair
\[
(-) \smashprod_{\ecal_{top}^H} \gcal_{top}^H \colon
\rightmod \ecal_{top}^H
\overrightarrow{\longleftarrow}
S_H \leftmod \colon \underhom (\gcal_{top}^H,-)
\]
is a strong symmetric monoidal Quillen equivalence.
\end{theorem}
\begin{proof}
If $\sigma \in \gcal_{top}^H$ is not $S_H$, then
it is cofibrant, so $\underhom (\sigma, -)$ preserves 
fibrations and all weak equivalences when considered as a functor from
$S_H \leftmod$ to $Sp^\Sigma_+$.
The functor $\underhom (S_H, -)$
preserves fibrations and all weak equivalences (since every object of
$S_H \leftmod$ is fibrant),
hence the above adjunction is a Quillen pair.

Following the proof of \cite[Theorem 3.9.3]{ss03stabmodcat}, 
we prove that the unit and counit of the derived adjunction
are weak equivalences. It suffices to do so on the generators
which are the free modules and the elements of $\gcal_{top}^H$. 
The free modules $F_\sigma = \hom(-, \sigma)$ are not cofibrant, however if we let 
$\cofrep \sphspec$ be a cofibrant replacement of $\sphspec$, the sphere spectrum 
in $Sp^\Sigma_+$, then $\cofrep \sphspec \smashprod F_\sigma$
is a cofibrant replacement of $F_\sigma$. The left
derived functor, $(-) \smashprod_{\ecal_{top}^H}^L \gcal_{top}^H$,
takes $\cofrep \sphspec \smashprod F_\sigma$ to $\cofrep \sphspec \smashprod \sigma$.
Since $\sigma$ is either $S_H$ or cofibrant, this
is weakly equivalent to $\sigma$.
The right adjoint preserves all weak equivalences and it follows that 
unit and counit of the derived adjunction
are weak equivalences. Hence we have a Quillen
equivalence, that this is a strong symmetric monoidal follows by the 
same arguments as for the proof of Theorem \ref{thm:FiniteAlgMorita}.
\end{proof}

\begin{remark}\label{rmk:whyfibrant}
In order to know that the above is a strong monoidal Quillen equivalence 
we need to know that $\gcal_{top}^H$ is closed under the smash product,
that every object is fibrant and that every non-unit object is cofibrant. 
In a general model category there is no reason to expect that the smash product
of fibrant objects will be fibrant.
Hence we use the category of $S_H$-modules
in EKMM $S$-modules, where every object is fibrant.
\end{remark}

\section{From $\ecal_{top}^H$ to $\ecal_t^H$.}\label{sec:ETOPtoET}

We now have two enriched categories
that we wish to compare:
$\ecal_{top}^H$ and $\ecal_a^H$. The difficulty is that the first is 
a category enriched over (positive) symmetric spectra, the second is enriched over 
rational chain complexes. 
Theorem \ref{thm:finiteEtopisEt} creates a new category 
$\ecal_t^H$ from $\ecal_{top}^H$, 
these categories have the same set of objects: $\gcal_{top}^H$, 
but $\ecal_t^H$ is enriched over
$\ch(\qq \leftmod)$. The categories of modules
$\rightmod \ecal_t^H$ and 
$\rightmod \ecal_{top}^H$ are Quillen equivalent as symmetric monoidal
model categories. 
Furthermore this construction comes with an isomorphism
of monoidal $\ch(\qq \leftmod)$-enriched categories between 
$\h_* \ecal_t^H$ and $\pi_* \ecal_{top}^H$. 

This theorem allows us to move from a topological
setting to an algebraic setting, while preserving monoidal
structures and also keeping control over the homology
of the new category. In terms of 
\cite[Example 5.1.2]{ss03stabmodcat},
see section \ref{sec:outline}, this theorem corresponds
to the zig-zag of Quillen equivalences
between $\rightmod \h \underline{A}$
and $\rightmod \underline{A}$. However this zig-zag 
passes through a model category without a monoidal 
product and so one cannot use it
to to understand monoidal structures. 
Additionally, this zig-zag only exists when 
the homotopy groups of 
the $Sp^\Sigma$-enriched category $\h \underline{A}$
are concentrated in degree zero, whereas our 
theorem has no such requirement.

Theorem \ref{thm:finiteEtopisEt} is similar to 
\cite[Corollary 2.16]{shiHZ}, which is stated below. 
Recall that a stable model category is called \textbf{rational} if the set of maps 
in the homotopy category between any two objects forms a rational vector space
(being stable implies that such sets are always abelian groups).

\begin{proposition}\label{prop:shipmachine}
Let $\ccal$ be a rational stable model category, 
with a set $\gcal$ of compact generators,
which is Quillen equivalent to a $Sp^\Sigma$-enriched model category. 
Then there a is category $\mathcal{A}$, 
which is enriched over rational
chain complexes and a chain of 
Quillen equivalences between $\ccal$ and the category of
right $\mathcal{A}$-modules. The objects of $\mathcal{A}$ correspond 
to the objects in $\gcal$ and there is an isomorphism
of graded $\qq$-categories between the homology category 
$\h_* \mathcal{A}$ and the full graded subcategory of
$\ho(\ccal)$ with objects $\gcal$. 
\end{proposition}

We introduce the Quillen pairs and model categories from 
\cite{shiHZ} that we will need for our construction. Note that the functor $L$ below
does not have a simple description so we concentrate on the remaining
functors.  

As well as $Sp^\Sigma$, we also use $Sp^\Sigma(\sqq \leftmod)$, symmetric spectra 
in simplicial $\qq$-modules see \cite{hov01}. 
For symmetric spectra in simplicial sets suspension is given in terms
of $S^1$, for  $Sp^\Sigma(\sqq \leftmod)$ suspension is defined using the object
$\widetilde{\qq} S^1$, which we define below.
We also use $Sp^\Sigma(\ch( \qq \leftmod)_+)$, 
symmetric spectra in non-negatively graded chain complexes
with suspension object $\qq[1]$, the chain complex consisting of $\qq$ in degree 1.
More generally we have $\qq[n]$, for $n$ an integer, which consists
of $\qq$ in degree $n$. 

The reduced free simplicial $\qq$-module functor from simplicial sets to
simplicial $\qq$-modules induces strong symmetric monoidal
Quillen adjunction below. Recall that reduced means that simplices of the form
$0 \cdot x$ and $q \cdot *$ are identified 
with the basepoint.
\[
\widetilde{\qq} \colon Sp^\Sigma \overrightarrow{\longleftarrow}
Sp^\Sigma(\sqq \leftmod) \colon U.
\]

Normalisation defines a functor $N \colon \sqq \leftmod \to \ch( \qq \leftmod)_+$,
in fact $N$ is the right adjoint of monoidal Quillen equivalence between these
two categories. 
The functor $N$ is not strong monoidal, so $N$ does not directly induce
a functor on the level of symmetric spectra. 
Take an object $X$ of $Sp^\Sigma(\sqq \leftmod)$ 
and apply $N$ levelwise to $X$. 
There is an isomorphism $\qq[1] \to N( \widetilde{\qq} S^1)$,
hence we have maps
\[
\phi_n \colon \qq [n] \overset{\cong}{\to} N( \widetilde{\qq} S^1)^{\otimes_n} 
\to N( (\widetilde{\qq} S^1)^{\otimes_n})
\]
which we use to obtain a collection of maps 
$\qq[n] \otimes NX_m \to NX_{n+m}$. 
This information assembles to give an object of 
$Sp^\Sigma(\ch( \qq \leftmod)_+)$, which we call $\phi^* N(X)$. 
This gives a functor $\phi^* N$ which is the right adjoint
of a symmetric monoidal Quillen equivalence. 
\[
L \colon Sp^\Sigma(\ch( \qq \leftmod)_+) \overrightarrow{\longleftarrow} 
Sp^\Sigma(\sqq \leftmod) \colon \phi^* N.
\]
Let $C_0 \colon \ch(\qq \leftmod) \to \ch( \qq \leftmod)_+$ be the functor
which takes a chain complex $X$ to its $[-1]$-connective cover. 
Thus $C_0 X_n$ is $X_n$ for $n > 0$, zero for $n < 0$ and 
$\ker (\partial \colon X_0 \to X_{-1})$ for $n=0$. 
For a chain complex $Y$ let $RY$ be the symmetric spectrum with 
$(RY)_n = C_0 (Y \otimes \qq[n])$.
This has a left adjoint $D$, thus we obtain the final adjoint pair which is a 
strong symmetric monoidal Quillen equivalence. 
\[
D \colon Sp^\Sigma(\ch( \qq \leftmod)_+)
\overrightarrow{\longleftarrow}
\ch( \qq \leftmod )\colon R
\]
We can give an explicit definition of $D$ taken from \cite{shiHZ}.
Let $I$ be the skeleton
of the category of finite sets and injections with objects $\mathbf{n}$. 
For $X \in Sp^\Sigma(\ch( \qq \leftmod)_+)$ define a
functor $D_X \colon I \to \ch(\qq \leftmod)$ by $D_X(\mathbf{n}) = \qq [-n] \otimes X_n$. 
There are structure maps
$\sigma \colon \qq[m - n] \otimes X_n \to X_m$ with adjoints 
$\widetilde{\sigma} \colon X_n \to \qq[n - m] \otimes X_m$. For a standard inclusion
of a subset $\alpha \colon \mathbf{n} \to \mathbf{m}$ the map $D_X(\alpha)$ is 
$\id_{\qq[-n]} \otimes \widetilde{\sigma}$. 
For an isomorphism in I, the action is
given by the tensor product of the action on $X_n$ and the sign action on $\qq[-n]$. 
The functor $D \colon Sp^\Sigma(\ch( \qq \leftmod)_+) \to \ch( \qq \leftmod )$ 
is defined by $DX = \colim_I D_X$. 

There has been some confusion over whether or not $D$ 
is \emph{symmetric} monoidal, see \cite{HZcorrection}. We assert that 
$D$ is indeed symmetric monoidal, 
this is based upon a detailed note written by Neil Strickland:
\cite{Dsymmon}. 

Now we can start constructing our new category $\ecal_t$. 
The category $\ecal_{top}^H$ is enriched over symmetric spectra
and we have given $\rightmod \ecal_{top}^H$ a model structure
using $Sp^\Sigma_+$. Let us temporarily call this
$(\rightmod \ecal_{top}^H)_+$.
We can also consider 
the category $\rightmod \ecal_{top}^H$ with model structure
defined via $Sp^\Sigma$ (that is, using the stable model structure). 
Since the weak equivalences are the same in both cases and the
fibrations are defined object-wise we have the following.
\begin{lemma}\label{lem:postostable}
The identity functor from 
$(\rightmod \ecal_{top}^H)_+$ to
$\rightmod \ecal_{top}^H$ is the left adjoint of a
Quillen equivalence. 
\end{lemma}
From here on, we only consider $\rightmod \ecal_{top}^H$ with the model
structure induced from $Sp^\Sigma$. This is a much better behaved category
as now the unit is cofibrant. 

Let $\widetilde{\qq} \ecal_{top}^H$ be the $Sp^\Sigma(\sqq \leftmod)$-enriched category
with object set $\gcal_{top}^H$ and morphisms
defined by $(\widetilde{\qq} \ecal_{top}^H)(a,b) = \widetilde{\qq} (\ecal_{top}^H(a,b))$.
This construction is a simplification of 
\cite[Proposition A.3b]{dugshi}.
Since $\widetilde{\qq}$ is symmetric monoidal, $\widetilde{\qq} \ecal_{top}^H$
is a symmetric monoidal enriched category. 
We repeat this operation twice more to form $D \phi^* N \tilde{\qq} \ecal_{top}^H$
a symmetric monoidal $\ch(\qq \leftmod)$-enriched category.
We define $\ecal_t^H$ as $D \phi^* N \tilde{\qq} \ecal_{top}^H$.

\begin{proposition}\label{prop:moduleadjunctions}
For each of the adjoint pairs
$(\tilde{\qq}, U)$, $(L,\phi^*N)$
and $(D,R)$,
there is an induced Quillen equivalence as below.
\begin{eqnarray*}
\tilde{\qq} : \rightmod \ecal_{top}^H
& \overrightarrow{\longleftarrow} &
\rightmod \tilde{\qq} \ecal_{top}^H : U' \\
L' : \rightmod\phi^* N \tilde{\qq}  \ecal_{top}^H
& \overrightarrow{\longleftarrow} &
\rightmod\tilde{\qq} \ecal_{top}^H : \phi^*N \\
D   : \rightmod \phi^* N \tilde{\qq} \ecal_{top}^H
& \overrightarrow{\longleftarrow} &
\rightmod D \phi^* N \tilde{\qq} \ecal_{top}^H : R'
\end{eqnarray*}
\end{proposition}
\begin{proof}
The induced adjunctions are defined
in \cite[Section 3]{ss03monequiv} and we give brief details below.
Since $\tilde{\qq}$ and $D$ are strong monoidal
these pass to the categories of modules
as above without change. The right adjoint $\phi^*N$
also passes directly to the module categories,
whereas all the other functors must be changed.
The right adjoints $U'$ and $R'$ are slight alterations of
$U$ and $R$, we demonstrate for $U'$. Take a
$\tilde{\qq} \ecal_{top}^H$-module $M$, 
then $U'M(\sigma) = U(M(\sigma))$,
we must then give maps
\[
U M(\sigma)  \smashprod \ecal_{top}^H(\sigma', \sigma)
\to U M(\sigma').
\]
We do so by applying the unit map
$\ecal_{top}^H(\sigma', \sigma) \to U \tilde{\qq} \ecal_{top}^H(\sigma', \sigma)$
and then using the monoidality of $U$ and the action map of $M$.
The left adjoint $L'$ is more complicated, since
it is not \emph{strong} monoidal.
Take a $\phi^* N \tilde{\qq} \ecal_{top}^H$-module
$M$, then $L'M(a)$ is defined as the coequaliser of the following diagram
(we describe the maps below).
\[
\bigvee_{b,c} L \big( M(b) \smashprod \phi^* N \tilde{\qq} \ecal_{top}^H(c,b) \big)
\smashprod \tilde{\qq} \ecal_{top}^H(a,c)
\overrightarrow{\longrightarrow}
\bigvee_{d} LM(d) \smashprod \tilde{\qq} \ecal_{top}^H(a,d)
\]
One map is induced by the action of
$\phi^* N \tilde{\qq} \ecal_{top}^H$ on $M$
and the other is the composite of the op-monoidal structure on $L$,
the counit of $(L, \phi^*N)$ and the composition map of
$\tilde{\qq} \ecal_{top}^H$.

The pair $(\tilde{\qq},U')$ induce a Quillen pair
between $\rightmod \ecal_{top}^H$ and
$\rightmod \tilde{\qq} \ecal_{top}^H$.
The free modules are a set of generators
for these categories. Since for each $a$ and $b$ in $\gcal_{top}^H$,
$\ecal_{top}^H(a,b)$ has rational homotopy groups, it follows that
the unit and counit
for the derived adjunctions
are equivalences on these generators and thus
$(\tilde{\qq},U')$ is a Quillen equivalence.
Since $(L, \phi^*N)$ and $(D,R)$ are Quillen equivalences
\cite[Theorem 6.5]{ss03monequiv}, implies that the other two pairs
Quillen equivalences.
\end{proof}

\begin{proposition}
The Quillen equivalences $(\tilde{\qq},U')$ and $(D,R')$ are
strong symmetric monoidal Quillen equivalences. The adjunction
$(L', \phi^*N)$ is a symmetric monoidal Quillen equivalence.
\end{proposition}
\begin{proof}
The first statement is a routine exercise in manipulating coends. 
The second requires some more work. It is easy to see that 
$\phi^*N$ induces a symmetric monoidal functor 
$\rightmod\tilde{\qq} \ecal_{top}^H \to 
\rightmod\phi^* N \tilde{\qq}  \ecal_{top}^H$. 

We must show that $\eta \colon L' (\phi^* N \tilde{\qq} \ecal_{top}^H(-, S_H))
\to \tilde{\qq} \ecal_{top}^H(-, S_H)$ is a weak equivalence
and that for cofibrant
$\phi^* N \tilde{\qq} \ecal_{top}^H$-modules $X$ and $Y$,
the map $m \colon L' (X \square Y) \to L'X \square L' Y$
(which exists since $\phi^*N$ is monoidal)
is a weak equivalence.
It suffices to prove the second condition for 
the free modules of $\rightmod \phi^* N \tilde{\qq} \ecal_{top}^H$,
since these generate the homotopy category. 
Thus, to prove both statements, 
we only need consider the behaviour of $L'$
on the free modules. 
Let $\textrm{sym} (\qq[1])$ denote the unit of 
$Sp^\Sigma(\ch( \qq \leftmod)_+)$ and 
$\textrm{sym} (\widetilde{\qq} S^1)$ be the unit  of
$Sp^\Sigma(\sqq \leftmod)$, these are both cofibrant.
Then 
\[
L' (\phi^* N \tilde{\qq} \ecal_{top}^H(-, a)) \cong
(\tilde{\qq} \ecal_{top}^H(-, a)) \otimes L \textrm{sym} (\qq[1])
\]
Recall the following three points: there is a natural isomorphism 
$F_g \square F_k \cong F_{g \smashprod k}$,
weak equivalences are defined levelwise in categories of right modules
and smashing with a cofibrant object of $Sp^\Sigma(\sqq \leftmod)$
preserves weak equivalences (see the proof of \cite[Corollary 3.4]{shiHZ}). 
Thus all we need prove is that 
$L \textrm{sym} (\qq[1]) \to \textrm{sym} (\widetilde{\qq} S^1)$ 
and 
$L \textrm{sym} (\qq[1]) \to L \textrm{sym} (\qq[1]) \otimes L \textrm{sym} (\qq[1])$
are weak equivalences. This result is part of \cite[Proposition 4.4]{shiHZ},
which states that the adjunction $(L, \phi^* N)$
is a monoidal Quillen equivalence.
\end{proof}

By \cite[Proposition 4.4, Lemma 4.8 and the proof of Theorem 1.2]{shiHZ}
the functors $\phi^* N$, $D$ and $\widetilde{\qq}$ preserve all weak equivalences.
Thus we do not need to consider derived functors
in the following work.

Now we give a monoidal isomorphism of categories
enriched over graded $\qq$-modules between 
$\h_* \ecal_t^H$ and $\pi_* \ecal^H_{top}$. 
Consider $\pi_* \ecal_{top}^H (a,b)$ which we define as 
$[S, \ecal_{top}^H (a,b)]_*^{\Sigma}$, 
graded maps in the homotopy category of symmetric spectra. 
We can apply the functor $\widetilde{\qq}$
to obtain a map as below, with the right hand side the set of  graded maps 
in the homotopy category of symmetric spectra in simplicial 
$\qq$-modules.
\[
[S, \ecal_{top}^H (a,b)]_*^{\Sigma}
\to 
[\widetilde{\qq}S, \widetilde{\qq}\ecal_{top}^H (a,b)]_*^{\sqq}
\]
We call the right hand side of the above $\pi_* \widetilde{\qq}\ecal_{top}^H (a,b)$.
This is an isomorphism since $\ecal_{top}^H (a,b)$
has rational homotopy groups. 
Furthermore this preserves the monoidal structures as we now explain. 
Taking the smash product of symmetric spectra
we obtain the following maps. 
\begin{eqnarray*}
[S, \ecal_{top}^H (a,b)]_*^{\Sigma} \otimes [S, \ecal_{top}^H (c,d)]_*^{\Sigma}
& \longrightarrow &
[S, \ecal_{top}^H (a,b) \smashprod \ecal_{top}^H (c,d)]_*^{\Sigma} \\
& \longrightarrow &
[S, \ecal_{top}^H (a \smashprod c,b \smashprod d) ]_*^{\Sigma}
\end{eqnarray*}
This defines the monoidal structure on $\pi_* \ecal_{top}^H$,
one defines a monoidal structure on $\pi_* \widetilde{\qq} \ecal_{top}^H$ similarly.
We can apply $\widetilde{\qq}$ to the various stages of the above
and obtain a large commuting diagram which implies that 
$\widetilde{\qq} \colon \pi_* \ecal_{top}^H \to \pi_* \widetilde{\qq} \ecal_{top}^H$ 
is an isomorphism of symmetric monoidal enriched categories.
We repeat this using $\phi^* N$ and $D$ noting that 
$D \phi^* N \widetilde{\qq} S$ is weakly equivalent to $\qq$
and that for any chain complex $X$, $\h_* X \cong [\qq,X]_*^{\ch}$
graded maps in the homotopy category of rational chain complexes. 
Thus we have obtained an isomorphism of symmetric monoidal enriched categories 
$\pi_* \ecal_{top}^H \to \h_* D \phi^* N \widetilde{\qq} \ecal_{top}^H$. 
This section is summarised in the following, 
see also \cite[Theorem 4.1]{greshi}.

\begin{theorem}\label{thm:finiteEtopisEt}
Let $\ecal_t^H$ denote the symmetric monoidal $\ch( \qq \leftmod)$-enriched 
category $D \phi^* N \tilde{\qq} \ecal_{top}^H$.
There is a zig-zag of monoidal Quillen equivalences between
$\rightmod \ecal_{top}^H$ (enriched over
$Sp^\Sigma_+$) and a category
$\rightmod \ecal_{t}^H$
(which is enriched over $\ch( \qq \leftmod)$).
This zig-zag induces
an isomorphism of monoidal graded $\qq$-categories:
$\pi_*(\ecal_{top}^H) \to \h_* \ecal_t^H$.
\end{theorem}

\begin{remark}\label{rmk:monoidalnoneqcase}
We consider the above theorem in the case of the trivial
group where our work reduces to that of
\cite{shiHZ}. Write $S_\qq$ for 
$S_{ \{e\} }$.
Here $\gcal_{top}$ has just one object
and $\rightmod \ecal_{top}$ is equivalent to
$S_\qq \leftmod$. Moving from
$\rightmod \ecal_{top}$ to $\rightmod \ecal_t$
is then just applying the functors of
\cite{shiHZ} to the spectrum $S_\qq$.
The resulting chain complex is then weakly equivalent
to $\qq$, as the comparison between
$\rightmod \ecal_t$ and $\rightmod \ecal_a$ below
will prove.
\end{remark}

We now have two $\ch(\qq \leftmod)$-enriched 
categories, $\ecal_t^H$ and $\ecal_a^H$.
While the definition of $\ecal_t^H$ is somewhat
complicated, we have control over its homology.

\section{Comparing $\ecal_t^H$ and $\ecal_a^H$}\label{sec:finitecomp}

We show that the homology of $\ecal_t^H$
is isomorphic, as an enriched category, to $\ecal_a^H$ in
Theorem \ref{thm:hocalc}. Furthermore,
this isomorphism respects the monoidal structures.
Theorem \ref{thm:finiteintrinsicformality}
implies that $\ecal_t^H$ and $\h_* \ecal_t^H$
are `quasi-isomorphic', thus by 
Corollaries \ref{cor:finiteETtoEA} and \ref{cor:monoidalcalc},
$\rightmod \ecal_t^H$, $\rightmod \h_* \ecal_t^H$
and $\rightmod \ecal_a^H$
are Quillen equivalent by strong symmetric monoidal 
adjunctions. 

This completes the main part of this paper,
as we have now shown that our algebraic model is
Quillen equivalent to the category of rational
$G$-spectra. We leave consideration of algebras
and modules over an algebra to the last section. 

We give a summary of \cite[Chapter XIX, Theorem 5.6]{may96} below,
we will use this result in our calculations. Recall that 
for a $G$-spectrum $X$, $\pi^H_*(X) \cong [G/H_+,X]^G_*$.
The right hand side admits an action of the Weyl group $W_G H= N_G H /H$ since
there is a $G$-map $G/H \times W_G H \to G/H$
given by $(gH,nH) \mapsto gn H$. Hence, 
$\pi^H_*(X) \otimes \qq$ is a $\qq W_G H$-module.
Recall the idempotent of the rational Burnside ring 
$e_H$ from section \ref{sec:finitetop},
then we have a $\qq W_G H$-module $\iota_H^*(e_H)_* \pi^H_*(X) \otimes \qq$. 
For the rest of this section we will write
$e_H$ for $\iota_H^*(e_H)_*$.
\begin{theorem}
For $G$-spectra $X$ and $Y$,
there is an isomorphism of graded rational vector spaces
\[
[X,Y]^G_\qq \cong \bigoplus_{(H) \leqslant G}
\hom_\qq(e_H \pi^H_*(X) \otimes \qq, e_H \pi^H_*(X) \otimes \qq)^{W_G H}.
\]
\end{theorem}

We can relate the above result to our work via the following isomorphisms. 
Recall the space $E \langle H \rangle$ which is the cofibre of the map
$E[<H]_+ \to E[\leqslant H]_+$ and is 
rationally equivalent to $e_H S$. 
We write $[X,Y]_\qq^{G|H}$ for maps in the homotopy category of
$L_{E \langle H \rangle} G \mcal_\qq$
and $\fibrep_H$ for fibrant replacement in this model category.
The $G$-spectra $\fibrep_H Y$ and $e_H Y$ are rationally equivalent, 
since $e_H Y$ is $E \langle H \rangle$-local, which
gives us the third isomorphism below.
\[
[X,Y]^G_\qq \cong 
\bigoplus_{(H) \leqslant G} [X,Y]_\qq^{G|H} \cong 
\bigoplus_{(H) \leqslant G} [X,\fibrep_H Y]_\qq^{G} \cong 
\bigoplus_{(H) \leqslant G} [X,e_H Y]_\qq^{G}
\]
If $X$ and $Y$ are $S_H$-modules then 
$\pi_*^H(X) \cong e_H \pi^H_*(X) \otimes \qq$
and $[X,Y]^G_\qq \cong [X,Y]^{S_H}_*$,
since every $S_H$-module is 
$S^0_\mcal \qq \smashprod E \langle H \rangle$-local
(Lemma \ref{lem:SHobject}). 
Thus, for $S_H$-modules $X$ and $Y$, 
we have an isomorphism of graded rational vector spaces
\[
[X,Y]^{S_H}_* \cong \hom_\qq(\pi^H_*(X), \pi^H_*(Y) )^{W_G H}.
\]

\begin{lemma}\label{lem:hogroupcalc}
For $gH \in W_G H$ we have a map 
$* \to G/H$ which sends the point to $gH$. 
Using the unit map of $S_H$, this induces a map of $G$-spectra 
$\widetilde{gH} \colon  S^0 \to (\cofrep G/H_+ ) \smashprod S_H$.
The assignment $gH \mapsto \widetilde{gH}$
induces an isomorphism
\[
\qq W_G H \to \pi_*^H ((\cofrep G/H_+ ) \smashprod S_H ).
\]
\end{lemma}
\begin{proof}
By the proof of Lemma \ref{lem:SHobject},
we can replace $S_H$ by $E \langle H \rangle \smashprod S^0_\mcal \qq$, 
so the above is isomorphic to 
$\pi_*(((\cofrep G/H_+ ) \smashprod E \langle H \rangle \smashprod S^0_\mcal \qq)^H)$.
Let $\fscr_H$ be the family of proper subgroups of $H$, then we have a cofibre
sequence of $H$-spaces $(E \fscr_H)_+ \to S^0 \to \widetilde{E }\fscr_H$.
Now $E[ \leqslant_G H]$
is $H$-equivariantly weakly equivalent to $S$,
so $E \langle H \rangle$ is
$H$-equivariantly weakly equivalent to
$\widetilde{E }\fscr_H$.
Thus we have an isomorphism (in fact one can define 
$\Phi^H$ as $(\widetilde{E }\fscr_H \smashprod (-))^H$)
\[
\pi_*(( \cofrep G/H_+  \smashprod E \langle H \rangle \smashprod S^0_\mcal \qq)^H) 
\cong \pi_*(\Phi^H ( \cofrep G/H_+  \smashprod S^0_\mcal \qq)).
\]
Since $\Phi^H$ commutes with smash products of 
cofibrant objects, cofibre sequences and coproducts
(such as those used to define $S^0_\mcal \qq$)
it follows that the above is isomorphic to 
$\pi_*(\Phi^H (\cofrep G/H_+ ) ) \otimes \qq$.
The following is standard:
$\Phi^H (\cofrep G/H_+ ) \simeq \Sigma^\infty (G/H^H)
= \Sigma^\infty W_G H$, the suspension spectrum of a finite set.
Thus we have proven that the groups
$\pi_*(\Phi^H (\cofrep G/H_+ )) \otimes \qq$,
$\pi_*(W_G H_+) \otimes \qq$ and $\qq W_G H$ are isomorphic.
Thus the groups $\pi_*^H ((\cofrep G/H_+ ) \smashprod S_H )$
and $\qq W_G H$ are isomorphic. 
It follows that the map specified in the lemma
is a particular choice of isomorphism. 
\end{proof}

The same method as 
above proves that the maps below are isomorphisms 
for $i \geqslant 0$. 
The second map is induced by 
the smash product of $S_H$-modules.
\[
(\qq W_G H)^{\otimes_i} \to 
(\pi_*^H (\cofrep G/H_+  \smashprod S_H ))^{\otimes_i}
\to \pi_*^H ((\cofrep G/H_+ )^{\smashprod_i} \smashprod S_H ).
\]
It follows that for every $i, j \geqslant 0$, 
we have an isomorphism $\alpha = \alpha_{i,j}$
induced by the smash product:
\[
\pi_*^H ((\cofrep G/H_+ )^{\smashprod_i} \smashprod S_H ) \otimes
\pi_*^H ((\cofrep G/H_+ )^{\smashprod_j} \smashprod S_H ) \to 
\pi_*^H ((\cofrep G/H_+ )^{\smashprod_{i+j}} \smashprod S_H ).
\]

Recall the $\gr(\qq \leftmod)$-enriched category $\pi_* \ecal_{top}^H$, 
for $a$ and $b$ in $\gcal_{top}^H$,
$\pi_* \ecal_{top}^H(a,b) = \pi_* (\underhom(a,b))$.
This category is symmetric monoidally isomorphic
to $\h_* \ecal_t^H$ (Theorem \ref{thm:finiteEtopisEt}).
We also have the $\gr(\qq \leftmod)$-enriched category on 
object set $\gcal_{top}^H$
with morphism object defined by 
$[a,b]^{S_H}_*$, graded maps in homotopy category of $S_H$-modules
(these morphism objects are rational 
since $S_H$ is fibrant in $G \mcal_\qq$).
This symmetric monoidal enriched category, which we call $\ho \ecal_{top}^H$, 
is isomorphic to $\pi_* \ecal_{top}^H$
via the adjunctions of section \ref{sec:finitetop}.
Furthermore, since these adjunctions are
symmetric monoidal, so is this isomorphism. 
For the purposes of calculations, it is easiest to work with 
objects of the form $[a,b]^{S_H}_*$.

\begin{proposition}
There is an isomorphism (specified in the proof below) 
of symmetric monoidal
$\gr(\qq \leftmod)$-enriched categories between 
$\ho \ecal_{top}^H$ and the full 
$\gr(\qq \leftmod)$-enriched category 
on the objects $\pi_*^H(a) \in \gr(\qq W_G H \leftmod)$, 
for $a \in \gcal_{top}^H$. We denote this category by 
$\pi_* \gcal_{top}^H$.
\end{proposition}
\begin{proof}
Thus for $a$ and $b$ in $\gcal_{top}^H$, we have 
$\pi_* \gcal_{top}^H (a,b) = \hom_\qq (\pi_*^H(a),\pi_*^H(b))^{W_G H}$.
Note that each $\pi_*^H(a) = [S,a]^H_*$ is concentrated in degree zero.
The forgetful functor induces a map $[a,b]^{S_H} \to [a,b]^{H}$,
combining this with composition gives a map of graded
$\qq W_G H$-modules
$
[a,b]^{S_H}_* \otimes_\qq [S,a]^H_* \to [S,b]^H_*.
$
This map has an adjoint, which is a map of graded 
$\qq$-modules
$
[a,b]^{S_H}_* \to \hom_\qq( [S,a]^H_* , [S,b]^H_*)^{W_G H}.
$
Thus we have a $\gr(\qq \leftmod)$-enriched functor
\[
\pi_*^H(-) = [S,-]^H_* \colon \ho \ecal_{top}^H \to \pi_* \gcal_{top}^H.
\]
This functor is an isomorphism of $\gr(\qq \leftmod)$-enriched categories, 
since (as proven above)
$[X,Y]^{S_H}_*$ and $\hom_\qq(\pi^H_*(X), \pi^H_*(Y) )^{W_G H}$
are isomorphic.
We now consider the diagram below, where everything is concentrated in degree zero. 
\[
\xymatrix{
[a,b]_*^{S_H} \otimes_\qq [c,d]_*^{S_H} \ar[r] \ar[d] &
\hom_\qq (\pi_*^H(a),\pi_*^H(b))^{W_G H} \otimes_\qq
\hom_\qq (\pi_*^H(c),\pi_*^H(d))^{W_G H} \ar[d] \\
[a \smashprod_{S_H} c, b \smashprod_{S_H} d]_*^{S_H} \ar[r]
&
\hom_\qq (\pi_*^H(a) \otimes_\qq \pi_*^H(c),\pi_*^H(b) \otimes_\qq \pi_*^H(d))^{W_G H}
}
\]
The lower horizontal map
uses the isomorphism $\alpha$ (constructed above)
from $\pi_*^H(a) \otimes_\qq \pi_*^H(c)$
to $\pi_*^H(a \smashprod_{S_H} c)$.
This diagram commutes since for maps $f$ and $g$ the smash product of 
$\pi_*^H(f)$ and $\pi_*^H(g)$
is $\pi_*^H(f \smashprod g)$. 
\end{proof}

\begin{theorem}\label{thm:hocalc}
There is a symmetric monoidal isomorphism 
of $\gr( \qq \leftmod)$-enriched 
categories $\ecal_a^H \to \h_* \ecal_{t}^H$. 
\end{theorem}
\begin{proof}
We replace 
$\h_* \ecal_{t}^H$ by 
$\pi_* \ecal_{top}^H$, which is isomorphic to
$\ho \ecal_{top}^H$ and hence to 
$\pi_* \gcal_{top}^H$. 

The category $\pi_* \gcal_{top}^H$ has object set
given by all smash products of $(\cofrep G/H_+ ) \smashprod S_H$ 
(here we mean the smash product of $S_H$-modules), with 
$S_H$ as the zero fold smash product. Similarly, 
$\ecal_a^H$ has object set given by all tensor products of
$\qq W_G H$ with $\qq$ as the zero fold tensor product.
Define an isomorphism on these object sets by sending
$\qq W_G H^{\otimes_i}$ to $(\cofrep G/H_+ )^{\smashprod_i} \smashprod S_H$
for $i \geqslant 0$.
In our work above, we have specified isomorphisms
\[
(\qq W_G H)^{\otimes_i} \to 
\left( \pi_*^H ((\cofrep G/H_+ ) \smashprod S_H ) \right)^{\otimes_i}
\to \pi_*^H ((\cofrep G/H_+ )^{\smashprod_i} \smashprod S_H ).
\]
These maps induce an isomorphism between
$\hom_\qq ((\qq W_G H)^{\otimes_i}, (\qq W_G H)^{\otimes_j} )^{W_G H}$
and 
\[
\hom_\qq \left( \pi_*^H ((\cofrep G/H_+ )^{\smashprod_i} \smashprod S_H),
\pi_*^H ((\cofrep G/H_+ )^{\smashprod_j} \smashprod S_H) \right)^{W_G H}.
\]
Thus we have an isomorphism
of $\gr( \qq \leftmod)$-enriched 
categories $\ecal_a^H \to \pi_* \gcal_{top}^H$.
Each of these isomorphisms of $\gr( \qq \leftmod)$-enriched 
categories is symmetric monoidal, so the result holds.
\end{proof}

We now know that $\h_* \ecal^H_{t}$ is concentrated in degree zero, 
our next result uses this to provide a comparison between
$\h_* \ecal^H_{t}$ and $\ecal^H_{t}$. 
Let $\ecal, \dcal$ be categories enriched over a model category, 
a map of enriched categories $\psi \colon \ecal \to \dcal$ 
is a \textbf{quasi-isomorphism} if 
it induces an isomorphism on the object sets and 
$\psi \colon \ecal(A,B) \to \dcal(\psi A, \psi B)$ is a 
weak equivalence for all pairs of objects $A$ and $B$.

\begin{theorem}\label{thm:finiteintrinsicformality}
If $\ecal$ is a $\ch( \qq \leftmod)$-category with
$\h_* \ecal$ concentrated in degree zero,
then $\ecal$ is quasi-isomorphic to $\h_* \ecal$
as $\ch( \qq \leftmod)$-categories. An explicit 
zig-zag is constructed below. 
If $\ecal$ is a symmetric monoidal enriched category then
the zig-zag consists of maps of symmetric monoidal
enriched categories.
\end{theorem}
\begin{proof}
We will create a $\ch( \qq \leftmod)$-enriched category $C_0 \ecal$
and a zig-zag of quasi-isomorphisms:
$\ecal \overset{\sim}{\longleftarrow} C_0 \ecal
\overset{\sim}{\longrightarrow} \h_0 \ecal = \h_* \ecal.$

As in section \ref{sec:ETOPtoET}, 
we will use a symmetric monoidal adjunction to construct $C_0 \ecal$. 
We have the $(-1)$-connective cover functor $C_0$, which 
is the right adjoint to the inclusion of $\ch( \qq\leftmod)_+$
into $\ch( \qq \leftmod)$. So for a chain complex $X$, 
$(C_0 X)_n$ is $X_n$ for $n > 0$, zero for $n < 0$
and is given by $\ker (\partial \colon X_0 \to X_{-1})$
for $n=0$. 
This adjunction is strong symmetric monoidal
and furthermore the counit is a symmetric monoidal natural transformation. 
Define $C_0 \ecal$ to have the same set of objects as $\ecal$ and let
$(C_0 \ecal)(a,b) = C_0 (\ecal(a,b))$. This is a symmetric monoidal 
$\ch( \qq\leftmod)$-enriched category. The counit gives
a map of symmetric monoidal $\ch( \qq\leftmod)$-enriched categories
$C_0 \ecal \to \ecal$. This follows from the  
commutative diagram below, where $X \otimes Y \to Z$ 
is a map in $\ch( \qq \leftmod)$.
\[
 \xymatrix{
C_0 X \otimes C_0 Y \ar[r] \ar[d] & C_0(X \otimes Y) \ar[r] \ar[dl] &
C_0 Z \ar[d] \\
X \otimes Y \ar[rr] && Z
} \]
For a chain complex of $\qq$-modules, $X$, we have a map $C_0 X \to \h_0 X$ which sends
$X_i$ to zero for $i > 0$ and sends
$(C_0(X))_0 =\ker(\partial_0) \to \h_0 X$ by the quotient.
We can consider $\h_0$ as a functor
$\ch(\qq \leftmod)_+ \to \qq \leftmod$,
this has a right adjoint
which includes $\qq \leftmod$ into $\ch( \qq \leftmod)_+$ by taking a
$\qq$-module $M$ to the chain complex
with $M$ in degree zero and zeroes elsewhere.
The map $C_0 X \to \h_0 X$ is induced by the
unit of this adjunction.
The functor $\h_0$ is monoidal, as is the inclusion of
$\qq \leftmod$ into $\ch( \qq \leftmod)_+$, thus we obtain
a symmetric monoidal $\ch( \qq \leftmod)$-category $\h_0 \ecal$.
Furthermore, the map $C_0 X \to \h_0 X$ is induced
by the unit of the adjunction which is a symmetric monoidal
natural transformation.
As above we obtain a map of symmetric monoidal 
$\ch( \qq \leftmod)$-enriched categories
$C_0 \ecal \to \h_0 \ecal= \h_* \ecal$,
which is a quasi-isomorphism.
\end{proof}

For a map of $\ch (\qq \leftmod)$-enriched categories
$\psi \colon \ecal \to \dcal$ there is an adjoint pair of
\textbf{extension and restriction of scalars} 
as defined in \cite[Section A.1]{ss03stabmodcat}. 
\[
- \otimes_\ecal \dcal \colon
\rightmod \ecal
\overrightarrow{\longleftarrow}
\rightmod \dcal \colon \psi^*
\]
The left adjoint is given in terms of a coend:
\[
(M \otimes_\ecal \dcal )(a)=
\int^a M(a) \otimes \dcal(-, \psi (a))
\]
the right adjoint simply lets 
$\ecal$ act on a $\dcal$-module via $\psi$. 
These form a Quillen pair that is a Quillen equivalence
if $\psi$ is a quasi-isomorphism by \cite[Theorem A.1.1]{ss03stabmodcat}.
If the map $\psi$ is symmetric monoidal, then it is routine to prove that
$(- \otimes_\ecal \dcal, \psi^*)$ is a strong symmetric monoidal
Quillen pair.

\begin{corollary}\label{cor:finiteETtoEA}
For each subgroup $H$, 
Theorem \ref{thm:finiteintrinsicformality}
specifies a zig-zag of symmetric monoidal quasi-isomorphisms of
$\ch( \qq \leftmod)$-categories.
\[
\ecal_t^H \overset{\sim}{\longleftarrow} C_0 \ecal_t^H
\overset{\sim}{\longrightarrow} \h_* \ecal_t^H 
\]
These quasi-isomorphisms induce a zig-zag of symmetric monoidal Quillen equivalences
of $\ch( \qq \leftmod)$-model categories.
\[
\rightmod \ecal_t^H
\overleftarrow{\longrightarrow} \rightmod C_0 \ecal_t^H
\overrightarrow{\longleftarrow} \rightmod \h_* \ecal_t^H
\]
\end{corollary}

We note that the above theorem and corollary are 
similar to results in \cite[Section 5]{shi02}. Now we prove that
the monoidal structures of 
$\ecal_a^H$ and $\h_* \ecal_{t}^H$ are equivalent, thus we can complete
our symmetric monoidal comparison between 
$\rightmod \ecal_a^H$ and $\rightmod \ecal_t^H$.

\begin{corollary}\label{cor:monoidalcalc}
The Quillen pair of extension and restriction of 
scalars gives a strong symmetric monoidal
Quillen equivalence (and an equivalence of categories) between 
$\ecal_a^H  \leftmod$ and $\h_* \ecal_{t}^H \leftmod$.
\end{corollary}

\section{Main Results}\label{sec:mainresults}

\begin{theorem}
For $G$ a finite group, there is a zig-zag of symmetric 
 monoidal Quillen equivalences 
between the category of rational $G$-equivariant spectra
and the algebraic model: $\prod_{(H) \leqslant G} \ch(\qq W_G H \leftmod)$.
\end{theorem}
\begin{proof}
We begin with Corollary \ref{cor:finitesplitting},
which splits rational $G$-spectra into the product
$\prod_{(H) \leqslant G} L_{E \langle H \rangle} G \mcal_\qq$.
Applying Proposition \ref{prop:GspecHtoSHmod}
to each factor of this category allows us to move to
$\prod_{(H) \leqslant G} S_H \leftmod$.
Next we use Theorem \ref{thm:finitemoritaequiv}
to move to modules over a $Sp^\Sigma$-enriched category,
$\prod_{(H) \leqslant G} \rightmod \ecal_{top}^H$.
We move to algebra, $\prod_{(H) \leqslant G} \rightmod \ecal_{t}^H$,
with Theorem \ref{thm:finiteEtopisEt},
and then use Corollary \ref{cor:finiteETtoEA}
and Corollary \ref{cor:monoidalcalc}
to get to the category
$\prod_{(H) \leqslant G} \rightmod \ecal_a^H$.
Finally we use Theorem \ref{thm:FiniteAlgMorita}
to complete the result.
\end{proof}

Since we have symmetric monoidal functors we can apply 
\cite[Theorem 3.12]{ss03monequiv} to each stage of the comparison
to obtain the following corollaries. The unit of $S_H \leftmod$
is not cofibrant, but this presents no difficulty.

\begin{corollary}
For each subgroup $H$, the above zig-zag induces a
zig-zag of Quillen equivalences between the category
of algebras in $S_H \leftmod$ and the category of
algebras in $\ch(\qq W_G H \leftmod)$. 
\end{corollary}

Let $i \colon C_0 \ecal_t^H \to \ecal_t^H $ and 
$p \colon C_0 \ecal_t^H \to \h_* \ecal_t^H$ be the maps constructed in 
Corollary \ref{cor:finiteETtoEA}.
Let $\psi \colon \ecal_a^H \to \h_* \ecal_t^H$ be the isomorphism constructed in 
Theorem \ref{thm:hocalc}, since this is an isomorphism
we can write $(\psi^{-1})^*$ for the left adjoint to 
$\psi^*$. 
The composites  $D \circ \phi^*N \circ \widetilde{\qq}$
and $U' \circ L' \cofrep \circ R'$ 
give a derived equivalence between $\ho( \ecal_{top}^H \leftmod)$ 
and $\ho(\ecal_t^H \leftmod)$ as stated in Theorem \ref{thm:finiteEtopisEt}.
Note that no cofibrant replacements ($\cofrep$) are needed in the
first of these composites as we are working rationally.
The functor $L'$ is the alteration of $L$ to modules over an enriched category
as we have described in section \ref{sec:ETOPtoET}. 
We will then need to alter $L'$ so that it acts on categories
of algebras, using the notation of 
\cite{ss03monequiv} we define $L'' = (L')^{mon}$.
We write out the derived composite functors needed for the 
next corollary. That such functors exist is perhaps
of more interest than the explicit
formulas, see also section \ref{sec:organise}.
The terms $\id \cofrep$ and $\id \fibrep$ are from 
the adjunction of Lemma \ref{lem:postostable}, where we change from the
$Sp^\Sigma_+$ to $Sp^\Sigma$.

\begin{definition}\label{def:majorderiv}
Let $\Theta$ be the derived functor
\[ (-) \otimes_{\ecal_a^H} \gcal_a^H \circ \psi^* \circ (\cofrep -) \otimes_{C_0 \ecal_t^H} \h_* \ecal_t^H 
\circ i^* \circ
D \circ \phi^*N \circ \widetilde{\qq} \circ \id \cofrep  \circ \underhom(\gcal_{top}^H, -) 
\] 
from $S_H \leftmod$ to $\qq W_G H \leftmod$. 
Let $\hh$ be the derived functor
\[
(\cofrep -) \smashprod_{\ecal_{top}^H} \gcal_{top}^H \circ \id \fibrep \circ
U' \circ L''\cofrep \circ R' \circ 
(\cofrep -) \otimes_{C_0 \ecal_t^H} \ecal_t^H \circ p^* \circ (\psi^{-1})^* \circ 
\underhom(\gcal_{a}^H, -) 
\]
from $\qq W_G H \leftmod$ to $S_H \leftmod$. 
\end{definition}

\begin{corollary}
For each $S_H$-algebra $A$ there is a zig-zag of Quillen equivalences
between $A \leftmod$ and $\Theta A \leftmod$.
For each $\qq W_G H$-algebra $B$ 
there is a zig-zag of Quillen equivalences
between $B \leftmod$ and $\hh B \leftmod$.
\end{corollary}

\begin{remark}
We note here that we have made no statement 
regarding model categories of
\emph{commutative} algebras.  
This is because not all of the model categories 
that we use have been shown to 
have model categories of commutative algebras. 
We expect that if this technical problem is solved, 
then our result will imply that the
model categories of commutative algebras in 
$G \mcal_\qq$ and commutative algebras in the algebraic model are
Quillen equivalent.

In particular, one would have to construct
a model category of commutative algebras in $\rightmod \ecal_{top}^H$,
when we are using the model structure arising from $Sp^\Sigma$
(see Lemma \ref{lem:postostable}). Recall that the usual model structure
of commutative algebras in $Sp^\Sigma$ is constructed from the
positive model structure, $Sp^\Sigma_+$. 
\end{remark}

\end{document}